\documentclass[a4paper]{amsart}

\usepackage{
   amsfonts,
   amssymb,
   amscd,
   latexsym,
   amsmath,
   amsbsy,
   enumerate,
   amsrefs,
   booktabs
}
\usepackage{tikz}
\usepackage{pgfplots}
\pgfplotsset{compat=1.18}
\usepackage{graphicx}
\usetikzlibrary{arrows.meta,decorations.pathmorphing,positioning}
\usepackage{mathrsfs}

\newtheorem{theorem}{Theorem}[section]

\newtheorem{lemma}[theorem]{Lemma}
\newtheorem{cor}[theorem]{Corollary}
\newtheorem{prop}[theorem]{Proposition}

\theoremstyle{definition}
\newtheorem{example}[theorem]{Example}
\newtheorem{remark}[theorem]{Remark}

\numberwithin{equation}{section}

\date{\today}

\title[Matrix Valued Orthogonal Polynomials and Random Walks]{Polynomial Transformations of Jacobi Matrices, Matrix valued Orthogonal polynomials and Random Walks}

\author{S. Mench\'on}

\address{IFEG-CONICET and FaMAF-Universidad Nacional de C\'ordoba, Ciudad Universitaria, C\'ordoba, Argentina;\\
Guangdong Technion Israel Institute of Technology, 241 Daxue Road, Jinping District, Shantou, Guandong Province, China.}
\email{silvia.menchon@unc.edu.ar, silvia.menchon@gtiit.edu.cn}

\author{P. Rom\'an}

\address{CIEM-CONICET and FaMAF-Universidad Nacional de C\'ordoba, Ciudad Universitaria, C\'ordoba, Argentina;\\
Guangdong Technion Israel Institute of Technology, 241 Daxue Road, Jinping District, Shantou, Guandong Province, China.}
\email{pablo.roman@unc.edu.ar, pablo.roman@gtiit.edu.cn}

\author{Y. Yin}
\address{Guangdong Technion Israel Institute of Technology, 241 Daxue Road, Jinping District, Shantou, Guandong Province, China.}
\email{yin12164@gtiit.edu.cn}

\begin{document}

\begin{abstract}
In this paper, we study matrix valued orthogonal polynomials whose block Jacobi
matrix is $\Theta(J)$, viewed in blocks of size $m=\deg\Theta$, where $J$ is a
tridiagonal Jacobi matrix and $\Theta$ is a scalar polynomial. We identify
their matrix measure of orthogonality intrinsically as the push-forward, under
$\Theta$, of a matrix measure constructed from the original scalar orthogonal
polynomials. We formulate reducibility in the push-forward variable and prove
that the $2\times2$ Hermite family associated with
$\Theta(x)=x^2+\alpha_1x$ is irreducible for $\alpha_1\neq0$ and reducible for
$\alpha_1=0$. This problem is motivated by random walks with transition
matrix $\Theta(J)$. For these models, we derive scalar and block
Karlin--McGregor representations for the powers of $\Theta(J)$. We apply the
construction to the mixed one- and two-ball Ehrenfest chain, obtaining
explicit spectral measures, spectral multiplicities, and transition formulas,
and we conclude with the corresponding $k$-ball and multi-ball extensions.
\end{abstract}

\maketitle

\section{Introduction}

In the Ehrenfest urn model, $N$ balls are distributed between two urns and, at
each discrete time step, one ball is chosen uniformly at random and moved to
the other urn. If the state records the number of balls in the first urn, the
resulting stochastic process is a nearest-neighbor Markov chain on
$\{0,\ldots,N\}$, and its transition matrix is tridiagonal. This is the setting
of the Karlin--McGregor theory \cite{karlin1959random}: the $\ell$-step
transition probabilities of a nearest-neighbor chain are recovered by
integrating $x^\ell$ against the orthogonality measure of the associated scalar
polynomials. For the Ehrenfest chain, these are the Krawtchouk polynomials
\cite{karlin_mcgregor_ehrenfest_1965}.

Suppose now that at each step the chain may perform either a one-ball or a
two-ball move. The state may then change by one or two units, or remain fixed,
and the transition matrix is five-diagonal. Since this matrix is no longer
tridiagonal, the direct scalar Karlin--McGregor description through a
three-term recurrence relation does not apply. Gr\"unbaum studied the equal
mixture of the one- and two-ball dynamics in \cite{grunbaum2009} and observed
that grouping consecutive states in pairs turns its transition matrix into a
block tridiagonal matrix with $2\times2$ blocks. This brings matrix valued
orthogonal polynomials naturally into the problem. The connection between
random walks with block tridiagonal transition matrices and matrix valued
orthogonal polynomials was considered, from different perspectives, in
\cites{dette_reuther_studden_zygmunt_2007,grunbaum2007random,grunbaum2009}.

The starting point of the present paper is a polynomial identity hidden in
this example. Let $M_0$ be the transition matrix of the classical Ehrenfest
chain and let $M_q$ be the transition matrix of the chain that performs a
two-ball move with probability $q$ and a one-ball move with probability $1-q$.
We prove in Section~\ref{sec:Krawtchouk-MV} that
\[
   M_q=\Theta(M_0),
   \qquad
   \Theta(x)=\frac{qN}{N-1}x^2+(1-q)x-\frac{q}{N-1}.
\]
Indeed, two successive one-ball moves select the same ball with probability
$1/N$, in which case the second move reverses the first. Otherwise, they select
two distinct balls and have the same net effect as a single two-ball move.
Hence, $M_q$ is a linear combination of $I$, $M_0$, and $M_0^2$. This
explains its five-diagonal structure and shows that the new dynamics is a
polynomial transform of the classical one. In particular, $M_q$ and $M_0$
share their eigenvectors, whereas distinct eigenvalues of $M_0$ may have the
same image under $\Theta$. The resulting spectral multiplicities therefore
depend on $q$.

The Ehrenfest identity suggests the following general construction. Let $\mu$
be a positive Borel measure with all moments finite, and let $(p_n)_{n\geq0}$ be a
sequence of scalar orthogonal polynomials with respect to $\mu$, normalized by
$p_0=1$. We write their three-term recurrence relation as
\[
   xp_n(x)=a_np_{n+1}(x)+b_np_n(x)+c_np_{n-1}(x),
   \qquad n\in\mathbb N_0,
\]
where $p_{-1}=0$, the recurrence coefficients are real, and
$a_nc_{n+1}>0$. These coefficients determine a tridiagonal Jacobi matrix $J$.
If
\[
   \mathbf p(x)=(p_0(x),p_1(x),\ldots)^{\mathsf T},
\]
then the recurrence relation is equivalent to
\[
   J\mathbf p(x)=x\mathbf p(x).
\]

Let $\Theta$ be a real scalar polynomial of degree $m$. The previous identity immediately gives
\[
   \Theta(J)\mathbf p(x)=\Theta(x)\mathbf p(x).
\]
The matrix $\Theta(J)$ has bandwidth $m$, and grouping its rows and columns in
consecutive blocks of size $m$ turns it into a block Jacobi matrix. By the
correspondence between higher-order recurrence relations and matrix valued
orthogonal polynomials established by Dur\'an and Van Assche
\cite{duran_vanassche_orthogonal_matrix_1995}, this block recurrence produces
a sequence $(P_n)_{n\geq0}$ of $m\times m$ matrix valued orthogonal
polynomials.

We also group the scalar polynomials into the vectors
\[
   \mathbf p_n(x)
   =
   \bigl(p_{mn}(x),\ldots,p_{mn+m-1}(x)\bigr)^{\mathsf T},
   \qquad n\in\mathbb N_0.
\]
In Section~\ref{sec:Construction:MVOP}, we prove the relation
\[
   P_n(\Theta(x))\mathbf p_0(x)=\mathbf p_n(x),
   \qquad n\in\mathbb N_0.
\]
This relation connects the scalar and matrix valued families and is used
throughout the paper to pass between them. After the substitution
$y=\Theta(x)$, it recovers each block of scalar polynomials from $P_n(y)$ and
the initial vector $\mathbf p_0(x)$. It therefore transfers many spectral and
orthogonality questions to the scalar setting. Reducibility, however, must be
studied intrinsically at the level of the resulting matrix measure.

The main goal of this paper is to identify this matrix measure of
orthogonality, determine when the corresponding family is irreducible, and
apply the construction to the random walks that motivate it.

Our first main result identifies the matrix measure of orthogonality. Set
\[
   W_p(x)=\mathbf p_0(x)\mathbf p_0(x)^\ast,
   \qquad
   d\Sigma(x)=W_p(x)\,d\mu(x).
\]
We prove that $(P_n)_{n\geq0}$ is orthogonal with respect to the push-forward
matrix measure
\[
   \Sigma_\Theta=\Theta_\#\Sigma
\]
and that the associated matrix inner product is non-degenerate. Suppose, in
addition, that $d\mu(x)=w(x)\,dx$, and set $W(x)=w(x)W_p(x)$. If
$W(x)\,dx$ has interval support $I$, then $\Sigma_\Theta$ is absolutely
continuous and, for almost every regular value $y$, its density is
\[
   \mathcal W(y)
   =
   \sum_{\substack{x\in I\\ \Theta(x)=y}}
   \frac{W(x)}{|\Theta'(x)|}.
\]
Thus the matrix structure is determined by the fibers of $\Theta$: the weight
at $y$ combines the contributions of the points in $I$ that are mapped to
$y$. We also give a finite version of the construction, in which the scalar
recurrence has a terminal relation and the matrix inner product is
non-degenerate on a truncated module of matrix polynomials. This is the
version required by the urn models, whose scalar measures have finite support.

Once the measure is known, it is natural to ask whether the family is
genuinely matrix valued or whether a constant change of basis decomposes
$\Sigma_\Theta$ into smaller blocks. Since orthogonality is expressed in the
variable $y=\Theta(x)$, reducibility must be formulated in this variable rather
than in the original variable $x$. We give a moment criterion that restricts
the possible reducing matrices and, under suitable hypotheses, forces them to
be real and diagonal. Once this has been established, a graph constructed from
the off-diagonal entries of $\mathcal W$ determines the remaining reducing
symmetries; in particular, the measure is irreducible precisely when this graph
is connected.

We illustrate the distinction between local properties of the density and
global reducibility with two explicit families. For the cubic map
$\Theta(x)=x^3-3x$ applied to a transported Legendre measure, the density has
pointwise rank one, two, and three on nonempty intervals, while the associated
$3\times3$ family remains irreducible. For Hermite input and
$\Theta(x)=x^2+\alpha_1x$, the $2\times2$ family is irreducible for
$\alpha_1\neq0$, whereas for $\alpha_1=0$ it decomposes into two scalar
Laguerre families.

We then return to the probabilistic side. When both $J$ and $\Theta(J)$ are
stochastic, we derive two Karlin--McGregor representations for the powers of
$\Theta(J)$: a scalar formula for the individual transition probabilities,
expressed in terms of $\mu$, and a block formula for the $m\times m$ blocks,
expressed in terms of the push-forward measure $\Sigma_\Theta$. These formulas
extend the classical scalar correspondence to the banded and block tridiagonal
transition matrices obtained as polynomial transforms of a stochastic Jacobi
matrix.

We apply these formulas to the mixed one- and two-ball Ehrenfest model. For the
equal mixture $q=1/2$ and $N$ odd, Gr\"unbaum described the block tridiagonal
recurrence, stated the eigenvalues and their multiplicities, established the
existence of a finite matrix measure of orthogonality, and represented its
masses through resolvent and terminal-polynomial formulas
\cite{grunbaum2009}*{pp.~272--275}. The identity $M_q=\Theta(M_0)$ extends the
spectral analysis to every $q\in[0,1]$: the eigenvectors are the Krawtchouk
vectors, the eigenvalues are obtained by applying $\Theta$ to the classical
spectrum, and the individual transition probabilities admit an explicit scalar
representation. For $q>0$ and $N$ odd, the block recurrence produces a genuine
finite family of $2\times2$ matrix valued orthogonal polynomials. We identify its
matrix measure as the push-forward of the finite Krawtchouk measure and obtain
the corresponding block transition formula. When $q=0$, the quadratic term
disappears and the construction reduces to the classical scalar Ehrenfest
problem.

We also determine the spectral multiplicities. Collisions occur only at
finitely many values of $q$; the smallest is $q=1/3$, so the spectrum is simple
for $0\leq q<1/3$, and at each exceptional value we count the double
eigenvalues explicitly. Finally, we extend the polynomial description to moves
involving more balls. We prove that the transition matrix of the pure $k$-ball
chain is a Krawtchouk polynomial evaluated at the classical Ehrenfest matrix.
Therefore, every finite mixture of such moves is again a polynomial in a
single tridiagonal transition matrix, and its spectral data are obtained by
evaluating the corresponding polynomial on the classical Ehrenfest spectrum.

The paper is organized as follows. In Section~\ref{sec:Construction:MVOP}, we
develop the scalar-to-matrix construction, prove the relation between the scalar
and matrix valued families, identify the push-forward matrix measure, derive its
inverse-branch density, and establish the finite version needed for the urn
models. Section~\ref{sec:irreducibility} formulates reducibility in the
push-forward variable, develops the moment and support-graph criteria, and
combines them with direct fiber arguments to analyze the cubic
transported-Legendre and quadratic Hermite families. In
Section~\ref{sec:Krawtchouk-MV}, we return to the motivating Ehrenfest model and
combine the finite construction with the scalar and block Karlin--McGregor
representations to obtain its explicit spectral solution and compare it with
\cite{grunbaum2009}. Section~\ref{sec:ehrenfest-extensions} extends the
polynomial description to the $k$-ball and multi-ball models. We conclude in
Section~\ref{sec:conclusions} with a summary of the results and several
directions for further study.

\section{Matrix valued orthogonal polynomials}
\label{sec:Construction:MVOP}
Following the correspondence between higher-order recurrence relations and
matrix valued orthogonal polynomials established by Dur\'an and Van Assche
\cite{duran_vanassche_orthogonal_matrix_1995}, in this section we develop the
scalar-to-matrix construction. Starting from a scalar measure $\mu$, with
Jacobi matrix $J$, and a scalar polynomial $\Theta$ of degree $m$, we view the
$(2m+1)$-diagonal matrix $\Theta(J)$ as a block Jacobi matrix with blocks of
size $m$. This produces a sequence of $m\times m$ matrix valued orthogonal
polynomials. We prove the scalar-to-matrix identity, identify the corresponding
matrix measure as a push-forward, and show that it defines a non-degenerate
matrix inner product. We then derive the inverse-branch density, analyze the
geometry of quadratic push-forwards, and conclude with the finite version of
the construction.

For $d\in\mathbb N$, we denote by
\[
   \mathcal M_d=\mathbb C^{d\times d},
\]
the algebra of $d\times d$ complex matrices, and by
\[
   \mathcal M_d[x]
   =\left\{
      P(x)=\sum_{j=0}^{n}A_jx^j:
      n\in\mathbb N_0,\ A_j\in\mathcal M_d
   \right\}
\]
the algebra of matrix polynomials. We use
this notation throughout for constant
matrices and matrix valued polynomials.

\subsection{Construction of the matrix valued polynomials}
\label{sec:MVOP}
Let $\mu$ be a positive Borel measure on $\mathbb{R}$ supported on the
closure of an interval $(a,b)$, where either endpoint may be infinite. We
assume that all moments of $\mu$ are finite. We denote by
$(p_n)_{n\geq0}$ the corresponding scalar orthogonal polynomials, normalized
by $p_0=1$. Their orthogonality relations are
\begin{equation}
\label{eq:scalar-orthogonality}
   \int_a^b p_n(x)p_k(x)\,d\mu(x)
   =\delta_{n,k}h_n,
   \qquad n,k\in\mathbb N_0,
\end{equation}
where $h_n>0$. The polynomials satisfy the three-term recurrence relation
\begin{equation}
\label{eq:three-term-scalar}
   xp_n(x)
   =a_np_{n+1}(x)+b_np_n(x)+c_np_{n-1}(x),
   \qquad n\in\mathbb N_0,
\end{equation}
with $p_{-1}=0$, $a_n,b_n,c_n\in\mathbb R$, and $a_nc_{n+1}>0$.

The Jacobi matrix $J$ is the infinite tridiagonal matrix whose diagonal
entries are $b_n$ and whose upper and lower off-diagonal entries are $a_n$
and $c_n$, respectively:
\begin{equation}
    \label{eq:TridiagMatrixJ}
   J =
   \begin{pmatrix}
      b_{0}  & a_{0}  & 0      & 0      & \cdots \\
      c_{1}  & b_{1}  & a_{1}  & 0      & \cdots \\
      0      & c_{2}  & b_{2}  & a_{2}  & \cdots \\
      0      & 0      & c_{3}  & b_{3}  & \cdots \\
      \vdots & \vdots & \vdots & \vdots & \ddots
   \end{pmatrix}.
\end{equation}
If we form the column vector of orthogonal polynomials,
\begin{equation}
\label{eq:vector_scalar_ps}
   \mathbf{p}(x) =
   \begin{pmatrix}
      p_{0}(x) \\
      p_{1}(x) \\
      p_{2}(x) \\
      \vdots
   \end{pmatrix},
\end{equation}
then the recurrence relation \eqref{eq:three-term-scalar} is equivalent to
\begin{equation}
\label{eq:Jp=xp}
   J\, \mathbf{p}(x) = x\, \mathbf{p}(x),
\end{equation}
Thus multiplication by $x$ on the basis $(p_n)_{n\geq0}$ is represented by the matrix $J$.

Let
\begin{equation}
\label{eq:theta}
   \Theta(x)
   =\alpha_mx^m+\alpha_{m-1}x^{m-1}+\cdots+\alpha_1x+\alpha_0,
\end{equation}
be a scalar polynomial with real coefficients and $\alpha_m\neq0$. Then
\[
   \Theta(J)
   =\alpha_mJ^m+\alpha_{m-1}J^{m-1}+\cdots+\alpha_1J+\alpha_0I,
\]
where $I$ denotes the identity matrix. Applying this polynomial to
\eqref{eq:Jp=xp}, we obtain
\begin{equation}
\label{eq:Jp=xp-theta}
   \Theta(J)\mathbf p(x)=\Theta(x)\mathbf p(x).
\end{equation}

Since $J$ is tridiagonal and $\alpha_m\neq0$, the matrix $\Theta(J)$ has
bandwidth $m$. Viewing its rows and columns in consecutive blocks of size
$m$, we obtain the block tridiagonal matrix
\begin{equation}
   \label{eq:block3diagmatrix}
   \Theta(J)=
   \begin{pmatrix}
      B_{0}  & A_{0}  & 0      & 0      & 0      & 0      \\
      C_{1}  & B_{1}  & A_{1}  & 0      & 0      & 0      \\
      0      & C_{2}  & B_{2}  & A_2    & 0      & 0      \\
      0      & 0      & C_{3}  & B_3    & A_3    & 0      \\
      \vdots & \vdots & \vdots & \vdots & \vdots & \ddots \\
   \end{pmatrix},
\end{equation}
where $A_i,B_i,C_i\in\mathcal M_m$.
\begin{remark}
\label{rmk:Ai-invertible}
For every $i\in\mathbb N_0$, the matrix $A_i$ is lower triangular and $C_i$
is upper triangular. Moreover,
\[
   (A_i)_{j,j}
   =\alpha_m a_{mi+j}\cdots a_{m(i+1)+j-1},
   \qquad 0\leq j\leq m-1.
\]
The condition $a_nc_{n+1}>0$ implies that $a_n\neq0$ for every $n$. Since
$\alpha_m\neq0$, every diagonal entry of $A_i$ is nonzero, and therefore
$A_i$ is invertible.
\end{remark}

The block recurrence and the invertibility of $A_n$ define a sequence
$(P_n)_{n\geq0}$ of $m\times m$ matrix polynomials by
\begin{equation}
\label{eq:polynomios-matriciales}
   xP_n(x)
   =A_nP_{n+1}(x)+B_nP_n(x)+C_nP_{n-1}(x),
   \qquad n\in\mathbb N_0,
\end{equation}
with $P_{-1}=0$ and $P_0=I$. Each $P_n$ has degree $n$ and an invertible
leading coefficient. We prove in Theorem~\ref{thm:orthogonalityPn} that these
polynomials are orthogonal with respect to the matrix inner product constructed
below. The sequence $(P_n)_{n\geq0}$ depends on both $\mu$ and $\Theta$.

\subsection{Relation between scalar and matrix valued polynomials}
We now prove the formula that relates scalar and matrix valued polynomials. For
$M\in\mathcal M_m$ and $0\leq j\leq m-1$, we denote by $M_{(j)}$ the $j$-th
row of $M$, with the rows numbered from $0$ to $m-1$. With this notation, the
$(mk+j)$-th row of the block tridiagonal matrix
\eqref{eq:block3diagmatrix} is
\begin{equation}
\label{eq:row-of-blok-tridiag}
   [0 \ \cdots \ 0 \,\, (C_k)_{(j)} \,\, (B_k)_{(j)} \,\, (A_k)_{(j)} \,\, 0 \ \cdots \ 0].
\end{equation}

We group the scalar polynomials into the vector polynomials
\begin{equation}
\label{eq:vector-p-scalar}
   \mathbf{p}_k(x)=\begin{pmatrix}
      p_{mk}(x) \\
      \vdots \\
      p_{mk+m-1}(x)
   \end{pmatrix},
   \qquad k\in\mathbb N_0,
\end{equation}
and set $\mathbf p_{-1}=0$. Taking the $(mk+j)$-th component in
\eqref{eq:Jp=xp-theta} and using \eqref{eq:row-of-blok-tridiag}, we obtain
\begin{align*}
   \Theta(x)p_{mk+j}(x)
   &=(A_k)_{(j)}\mathbf p_{k+1}(x)
     +(B_k)_{(j)}\mathbf p_k(x)
     +(C_k)_{(j)}\mathbf p_{k-1}(x),
\end{align*}
for $0\leq j\leq m-1$. Hence the vector polynomials satisfy
\begin{equation}
\label{eq:recurrence_vector_polys}
   \Theta(x)\mathbf p_k(x)
   =A_k\mathbf p_{k+1}(x)+B_k\mathbf p_k(x)+C_k\mathbf p_{k-1}(x),
   \qquad k\in\mathbb N_0.
\end{equation}
\begin{prop}
\label{prop:matrix-to-scalar}
The matrix polynomials $(P_n)_{n\geq0}$ satisfy
\begin{equation}
\label{eq:matrix-to-scalar}
   P_n(\Theta(x))\mathbf p_0(x)=\mathbf p_n(x),
   \qquad n\in\mathbb N_0.
\end{equation}
\end{prop}
\begin{proof}
We proceed by induction on $n$. For $n=0$, equation
\eqref{eq:matrix-to-scalar} follows from $P_0=I$. Assume that it holds for
$n=k$ and $n=k-1$, where the case $k=0$ uses $P_{-1}=0$ and
$\mathbf p_{-1}=0$. Multiplying the identity for $n=k$ by $\Theta(x)$ and
applying the matrix recurrence \eqref{eq:polynomios-matriciales} with its
variable evaluated at $\Theta(x)$, we obtain
\[
   \Theta(x)\mathbf p_k(x)
   =A_kP_{k+1}(\Theta(x))\mathbf p_0(x)
    +B_k\mathbf p_k(x)+C_k\mathbf p_{k-1}(x).
\]
On the other hand, the vector recurrence
\eqref{eq:recurrence_vector_polys} gives
\[
   \Theta(x)\mathbf p_k(x)
   =A_k\mathbf p_{k+1}(x)
    +B_k\mathbf p_k(x)+C_k\mathbf p_{k-1}(x).
\]
Subtracting these identities yields
\[
   A_k\left(
      P_{k+1}(\Theta(x))\mathbf p_0(x)-\mathbf p_{k+1}(x)
   \right)=0.
\]
Since $A_k$ is invertible by Remark~\ref{rmk:Ai-invertible}, equation
\eqref{eq:matrix-to-scalar} follows for $n=k+1$. This completes the induction
and proves the proposition.
\end{proof}

\subsection{Matrix inner product and orthogonality}

The scalar-to-matrix identity \eqref{eq:matrix-to-scalar} allows us to
construct a matrix inner product for the polynomials $(P_n)_{n\geq0}$. Set
\[
   W_p(x)=\mathbf p_0(x)\mathbf p_0(x)^\ast.
\]
For $P,Q\in\mathcal M_m[x]$, define
\begin{equation}
\label{eq:theta-inner-product}
   \langle P,Q\rangle_\Theta
   =\int_a^b
      P(\Theta(x))W_p(x)Q(\Theta(x))^\ast\,d\mu(x).
\end{equation}
Here, $^\ast$ denotes conjugate transpose. Thus
$\langle\cdot,\cdot\rangle_\Theta$ maps
$\mathcal M_m[x]\times\mathcal M_m[x]$ into $\mathcal M_m$.
The following theorem proves that this inner product is non-degenerate and
establishes the orthogonality of $(P_n)_{n\geq0}$.

\begin{theorem}
\label{thm:orthogonalityPn}
The map defined by \eqref{eq:theta-inner-product} is a non-degenerate matrix
inner product on $\mathcal M_m[x]$. More precisely, for
$T\in\mathcal M_m$ and $P,Q,R\in\mathcal M_m[x]$,
\begin{align*}
   \langle TP+Q,R\rangle_\Theta
   &=T\langle P,R\rangle_\Theta+\langle Q,R\rangle_\Theta,\\
   \langle P,Q\rangle_\Theta
   &=\langle Q,P\rangle_\Theta^\ast,\\
   \langle P,P\rangle_\Theta&\geq0,
   \qquad
   \langle P,P\rangle_\Theta=0\ \Longleftrightarrow\ P=0.
\end{align*}
Moreover,
\begin{equation}
\label{eq:matrix-orthogonality}
   \langle P_j,P_k\rangle_\Theta
   =\delta_{j,k}H_j,
   \qquad
   H_j=\operatorname{diag}
   \bigl(h_{mj},h_{mj+1},\ldots,h_{mj+m-1}\bigr),
\end{equation}
for $j,k\in\mathbb N_0$.
\end{theorem}
\begin{proof}
The linearity and symmetry properties follow directly from
\eqref{eq:theta-inner-product}. By Proposition~\ref{prop:matrix-to-scalar},
the orthogonality relations \eqref{eq:scalar-orthogonality} give
\begin{align*}
   \langle P_j,P_k\rangle_\Theta
   &=\int_a^b
      P_j(\Theta(x))W_p(x)P_k(\Theta(x))^\ast\,d\mu(x)\\
   &=\int_a^b\mathbf p_j(x)\mathbf p_k(x)^\ast\,d\mu(x)
     =\delta_{j,k}H_j,
\end{align*}
which proves \eqref{eq:matrix-orthogonality}.

It remains to prove positive definiteness. Let $P\in\mathcal M_m[x]$ have
degree $n$. Since $P_k$ has degree $k$ and an invertible leading coefficient,
the polynomials $P_0,\ldots,P_n$ form a basis, over $\mathcal M_m$ on the
left, of the matrix polynomials of degree at most $n$. Hence there are unique
matrices $\Gamma_k\in\mathcal M_m$ such that
\[
   P(x)=\sum_{k=0}^n\Gamma_kP_k(x).
\]
Using \eqref{eq:matrix-orthogonality}, we obtain
\begin{equation}
\label{eq:inner-product-P-P-Gammas}
   \langle P,P\rangle_\Theta
   =\sum_{k=0}^n\Gamma_kH_k\Gamma_k^\ast.
\end{equation}
Each $H_k$ is positive definite, so every summand on the right-hand side of
\eqref{eq:inner-product-P-P-Gammas} is positive semidefinite. Therefore
$\langle P,P\rangle_\Theta\geq0$. If $\langle P,P\rangle_\Theta=0$, the
sum of these positive semidefinite matrices can vanish only if every summand
vanishes. Since $H_k$ is positive definite, this implies $\Gamma_k=0$ for
$0\leq k\leq n$, and hence $P=0$. The converse is immediate.
\end{proof}

\subsection{The push-forward matrix measure}
\label{sec:pushforward-measure}

The inner product \eqref{eq:theta-inner-product} is written in the original
scalar variable $x$. We now express it intrinsically in the matrix spectral
variable $y=\Theta(x)$. Define the positive matrix measure $\Sigma$ on
$\mathbb R$ by
\begin{equation}
\label{eq:def-matrix-measure-Sigma}
   d\Sigma(x)=W_p(x)\,d\mu(x).
\end{equation}
Its push-forward under $\Theta$ is the matrix measure
\begin{equation}
\label{eq:def-pushforward-measure}
   \Sigma_\Theta=\Theta_\#\Sigma,
   \qquad
   \Sigma_\Theta(E)
   =\int_{\Theta^{-1}(E)}W_p(x)\,d\mu(x),
\end{equation}
for every Borel set $E\subseteq\mathbb R$.

\begin{prop}
\label{prop:pushforward-measure}
The matrix inner product \eqref{eq:theta-inner-product} can be written as
\begin{equation}
\label{eq:pushforward-inner-product}
   \langle P,Q\rangle_\Theta
   =\int_{\mathbb R}P(y)\,d\Sigma_\Theta(y)\,Q(y)^\ast.
\end{equation}
If, in addition, $d\mu(x)=w(x)\,dx$, set $W(x)=w(x)W_p(x)$ and assume that
the matrix measure $W(x)\,dx$ has support $I=[a,b]$. Set
\[
   I_\Theta=\overline{\Theta(I)}=[c,d],
\]
Then $\Sigma_\Theta$ has
support $I_\Theta$, is absolutely continuous, and
\[
   d\Sigma_\Theta(y)=\mathcal W(y)\,dy,
\]
where $\mathcal W(y)=0$ outside $[c,d]$ and, for almost every regular value
$y\in[c,d]$,
\begin{equation}
\label{eq:pushforward-density}
   \mathcal W(y)
   =\sum_{\substack{x\in I\\ \Theta(x)=y}}
     \frac{W(x)}{|\Theta'(x)|}.
\end{equation}
Moreover,
\begin{equation}
\label{eq:pushforward-density-integrability}
   \int_c^d\operatorname{tr}\mathcal W(y)\,dy
   =\int_a^b\operatorname{tr}W(x)\,dx
   =\sum_{j=0}^{m-1}h_j<\infty.
\end{equation}
In particular, all the entries of $\mathcal W$ are integrable on $[c,d]$.
\end{prop}

\begin{proof}
By the defining property of the push-forward measure, applied entrywise,
\[
   \int_{\mathbb R}P(y)\,d\Sigma_\Theta(y)\,Q(y)^\ast
   =\int_a^b
      P(\Theta(x))W_p(x)Q(\Theta(x))^\ast\,d\mu(x).
\]
Equation \eqref{eq:pushforward-inner-product} now follows from
\eqref{eq:theta-inner-product}.

Since $\Theta$ is continuous, $\Sigma_\Theta$ has support $I_\Theta=[c,d]=\Theta(I)$.
Remove the finitely many critical points of $\Theta$ and split the interior of $I$ into intervals $I_r$ on
which $\Theta$ is injective. If $\phi_r$ denotes the inverse branch with
values in $I_r$, ordinary change of variables gives the contribution
\[
   \frac{W(\phi_r(y))}{|\Theta'(\phi_r(y))|},
\]
on $\Theta(I_r)$. The inverse image of each critical value is finite and
hence has zero $\mu$-measure because $\mu$ is absolutely continuous.
Therefore these branch contributions account for the whole push-forward
measure, proving \eqref{eq:pushforward-density}.

Applying the same change of variables to the nonnegative scalar function
$\operatorname{tr}W$ yields the first equality in
\eqref{eq:pushforward-density-integrability}. Since
$W=w\mathbf p_0\mathbf p_0^\ast$, scalar orthogonality gives
\[
   \int_a^b\operatorname{tr}W(x)\,dx
   =\sum_{j=0}^{m-1}\int_a^b |p_j(x)|^2\,d\mu(x)
   =\sum_{j=0}^{m-1}h_j.
\]
The matrix $\mathcal W(y)$ is positive semidefinite almost everywhere, and
therefore
\[
   |\mathcal W_{ij}(y)|
   \leq\frac{\mathcal W_{ii}(y)+\mathcal W_{jj}(y)}{2}.
\]
Hence every entry is integrable.
\end{proof}

The three-term recurrence \eqref{eq:polynomios-matriciales} shows, in
addition, that the block Jacobi matrix associated with $\Sigma_\Theta$ is
\eqref{eq:block3diagmatrix}, namely $\Theta(J)$ viewed in blocks of size $m$.

\subsection{Quadratic push-forwards and the $2\times2$ geometry}
\label{subsec:quadratic-pushforward-geometry}

We now specialize the density formula \eqref{eq:pushforward-density} to
$m=2$ and the monic quadratic map
\begin{equation}
\label{eq:quadratic-map-branches}
   \Theta(x)=x^2+\alpha_1x+\alpha_0
   =(x-x_v)^2+\Theta(x_v),
   \qquad x_v=-\frac{\alpha_1}{2}.
\end{equation}
The reflection across the vertex,
\begin{equation}
\label{eq:quadratic-involution}
   \iota(x)=-x-\alpha_1,
\end{equation}
satisfies $\iota^2(x)=x$ and $\Theta(\iota(x))=\Theta(x)$.
Throughout this subsection, \(\sqrt{\cdot}\) denotes the principal square root. Since all arguments of the square roots below are nonnegative real numbers, this is the nonnegative real square root.
For $y\geq\Theta(x_v)$, the two inverse branches are
\begin{equation}
\label{eq:quadratic-inverse-branches}
   \phi_\pm(y)
   =\frac{-\alpha_1\pm
   \sqrt{\alpha_1^2+4(y-\alpha_0)}}{2}
   =x_v\pm\sqrt{y-\Theta(x_v)},
\end{equation}
and they are exchanged by $\iota$. In particular,
\begin{equation}
\label{eq:quadratic-jacobian}
   \left|\Theta'(\phi_\pm(y))\right|
   =2\sqrt{y-\Theta(x_v)}.
\end{equation}

Assume that $d\mu(x)=w(x)\,dx$ on an interval $(a,b)$ and write
$W(x)=w(x)W_p(x)$ for the matrix density before push-forward and
$d\Sigma_\Theta(y)=\mathcal W(y)\,dy$ for the weight in the spectral variable.
If the vertex lies outside $(a,b)$, only one inverse branch belongs to the
support.
For example, if $x_v<a$, then
\begin{equation}
\label{eq:quadratic-one-branch-weight}
   \mathcal W(y)
   =\frac{W(\phi_+(y))}
   {2\sqrt{y-\Theta(x_v)}},
   \qquad \Theta(a)<y<\Theta(b).
\end{equation}
The case $x_v>b$ is analogous, with the minus branch and the endpoints in
the corresponding order.

Suppose instead that $a<x_v<b$ and that
$\iota(a)\in[a,b]$; the configuration is shown in
Figure~\ref{fig:quadratic-pushforward-geometry}. Set
\[
   \mathcal I_1=[\Theta(x_v),\Theta(a)],
   \qquad
   \mathcal I_2=[\Theta(a),\Theta(b)].
\]
Both branches lie in the support for $y$ in the interior of
$\mathcal I_1$, whereas only $\phi_+$ remains for $y$ in the interior of
$\mathcal I_2$. At the critical value $\Theta(x_v)$ the two branches
coincide, while at the common endpoint $\Theta(a)$ both branches still
belong to the support. These endpoint distinctions do not affect the density,
which is defined almost everywhere. Thus
\begin{equation}
\label{eq:quadratic-two-branch-weight}
   \mathcal W(y)=
   \begin{cases}
   \displaystyle
   \frac{W(\phi_-(y))+W(\phi_+(y))}
        {2\sqrt{y-\Theta(x_v)}},
      & \Theta(x_v)<y < \Theta(a),\\[10pt]
   \displaystyle
   \frac{W(\phi_+(y))}
        {2\sqrt{y-\Theta(x_v)}},
      & \Theta(a)<y<\Theta(b).
   \end{cases}
\end{equation}
If $x_v$ is closer to $b$, the same formula holds after exchanging the two
branches and the roles of the endpoints. The Jacobian in
\eqref{eq:quadratic-jacobian} contributes the locally integrable factor
$(y-\Theta(x_v))^{-1/2}$. If $W$ is locally bounded near $x_v$, this is the
only singularity at the critical value. In general, integrability follows
from \eqref{eq:pushforward-density-integrability}.

For the Hermite $2\times2$ family studied in
Subsection~\ref{subsec:hermite-irreducibility}, the support is all of
$\mathbb R$, so both branches in \eqref{eq:quadratic-inverse-branches}
contribute for every $y>\Theta(x_v)$. Their sum gives the Hermite weight
\eqref{eq:hermite-pushforward-weight}. When $\alpha_1=0$, reflection across
the vertex preserves the scalar Hermite weight and the off-diagonal
contributions cancel, producing the reducible case. For $\alpha_1\neq0$,
this symmetry is lost and the family is irreducible, as proved in
Theorem~\ref{thm:hermite-irreducibility}.

\begin{figure}[t]
  \centering
  \begin{tikzpicture}
    \begin{axis}[
        axis lines=middle,
        xmin=-6, xmax=6,
        ymin=-1, ymax=45,
        xlabel={$x$}, ylabel={$y$},
        domain=-6:5,
        samples=300,
        enlargelimits=false,
        clip=false,
        xtick=\empty,
        ytick=\empty
      ]
      \addplot[thick] {x^2 + 2*x + 4};
      \addplot[ultra thick, color=magenta!70!blue]
        coordinates {(-4,0) (2,0)};
      \addplot[ultra thick, color=blue!70]
        coordinates {(2,0) (4,0)};
      \addplot[only marks, mark=*]
        coordinates {(-4,0) (-2.5,0) (0.5,0) (2,0) (4,0)};
      \node[below] at (axis cs:-4,0) {$a$};
      \node[below] at (axis cs:-2.5,0) {$\iota(x)$};
      \node[below] at (axis cs:0.5,0) {$x$};
      \node[below] at (axis cs:2,0) {$\iota(a)$};
      \node[below] at (axis cs:4,0) {$b$};
      \addplot[only marks, mark=*] coordinates {(-1,0)};
      \node[below] at (axis cs:-1,0) {$x_v$};
      \addplot[dotted, thick] coordinates {(-4,0) (-4,12)};
      \addplot[dotted, thick] coordinates {(-2.5,0) (-2.5,5.25)};
      \addplot[dotted, thick] coordinates {(0.5,0) (0.5,5.25)};
      \addplot[dotted, thick] coordinates {(2,0) (2,12)};
      \addplot[dotted, thick] coordinates {(4,0) (4,28)};
      \addplot[dotted, thick] coordinates {(-1,0) (-1,3)};
      \addplot[ultra thick, color=magenta!70!blue, domain=-4:2,
        samples=200] {x^2 + 2*x + 4};
      \addplot[ultra thick, color=blue!70, domain=2:4, samples=120]
        {x^2 + 2*x + 4};
      \addplot[dotted, thick] coordinates {(-4,12) (2,12)};
      \addplot[ultra thick, color=orange!85!black]
        coordinates {(0,3) (0,12)};
      \addplot[ultra thick, color=green!60!black]
        coordinates {(0,12) (0,28)};
      \node[left] at (axis cs:0,7.5) {$\mathcal I_1$};
      \node[left] at (axis cs:0,20) {$\mathcal I_2$};
    \end{axis}
  \end{tikzpicture}
  \caption{Scheme of a quadratic push-forward. The map $\Theta$ identifies
  $x$ with $\iota(x)=-x-\alpha_1$. The intervals $[a,x_v]$ and
  $[x_v,\iota(a)]$ both map onto $\mathcal I_1$, whereas
  $[\iota(a),b]$ maps onto $\mathcal I_2$. Thus every point of
  $\mathcal I_1$ has two inverse images in $[a,b]$, while every point of
  $\mathcal I_2$ has only one.}
  \label{fig:quadratic-pushforward-geometry}
\end{figure}

\subsection{Finite scalar-to-matrix construction}
\label{subsec:finite_sequences}
When the scalar measure has finite support, the construction above requires
two modifications. The scalar recurrence has a terminal relation, and the
resulting matrix inner product is non-degenerate only on a truncated module
of matrix polynomials. We now formulate this finite version. Let
\begin{equation}
\label{eq:finite-scalar-measure}
   \mu=\sum_{\ell=0}^{L-1}\omega_\ell\,\delta_{\nu_\ell},
   \qquad
   \nu_0,\ldots,\nu_{L-1}\in\mathbb R\ \text{distinct},
   \qquad
   \omega_\ell>0,
\end{equation}
and let $(p_n)_{0\leq n\leq L-1}$ be the corresponding scalar orthogonal
polynomials, normalized by $p_0=1$. Thus
\begin{equation}
\label{eq:finite-scalar-orthogonality}
   \sum_{\ell=0}^{L-1}
   p_n(\nu_\ell)p_s(\nu_\ell)\omega_\ell
   =h_n\delta_{n,s},
   \qquad 0\leq n,s\leq L-1,
\end{equation}
with $h_n>0$. The polynomials $p_n$ satisfy a three-term recurrence relation:
\begin{equation}
\label{eq:finite-scalar-recurrence}
   x p_n(x)=a_np_{n+1}(x)+b_np_n(x)+c_np_{n-1}(x),
   \qquad p_{-1}=0, \qquad 0\leq n\leq L-2,
\end{equation}
Here the index $n$ only runs from $0$ to $L-2$.
We may complete the index $n=L-1$ in the recurrence with a polynomial
$p_L$ of degree $L$ that vanishes at every node $\nu_\ell$.
Equivalently,
\begin{equation}
\label{eq:finite-scalar-terminal-recurrence}
   \nu_\ell p_{L-1}(\nu_\ell)
   =b_{L-1}p_{L-1}(\nu_\ell)
   +c_{L-1}p_{L-2}(\nu_\ell),
   \qquad 0\leq\ell\leq L-1.
\end{equation}
The finite Jacobi
matrix $J_L$ is the $L\times L$ matrix of multiplication by $x$ on the
functions supported on the nodes, written in the basis
$p_0,\ldots,p_{L-1}$. In particular,
\begin{equation}
\label{eq:finite-J-eigenvectors}
   J_L\mathbf p(\nu_\ell)=\nu_\ell\mathbf p(\nu_\ell),
   \qquad
   \mathbf p(x)=
   \begin{pmatrix}p_0(x)&\cdots&p_{L-1}(x)\end{pmatrix}^{\mathsf T}.
\end{equation}

Let $\Theta$ have degree $m$ and assume that
\begin{equation}
\label{eq:finite-dimension-divisibility}
   L=mr
\end{equation}
for a positive integer $r$. Grouping the rows and columns of
$\Theta(J_L)$ into $r\times r$ blocks of size $m\times m$ gives
\begin{equation}
\label{eq:finite-block-jacobi}
   \Theta(J_L)=
   \begin{pmatrix}
      B_0&A_0&&&\\
      C_1&B_1&A_1&&\\
         &\ddots&\ddots&\ddots&\\
         &&C_{r-2}&B_{r-2}&A_{r-2}\\
         &&&C_{r-1}&B_{r-1}
   \end{pmatrix}.
\end{equation}
The bandwidth argument used for \eqref{eq:block3diagmatrix} proves the block
tridiagonal form. Moreover, the same triangular-path argument as in
Remark~\ref{rmk:Ai-invertible} shows that $A_k$ is invertible for
$0\leq k\leq r-2$. There is no outgoing block $A_{r-1}$.

Define
\begin{equation}
\label{eq:finite-scalar-blocks}
   \mathbf p_k(x)=
   \begin{pmatrix}
      p_{mk}(x)\\ \vdots\\ p_{mk+m-1}(x)
   \end{pmatrix},
   \qquad 0\leq k\leq r-1,
\end{equation}
and construct the finite family $(P_k)_{0\leq k\leq r-1}$ by $P_{-1}=0$,
$P_0=I_m$, and
\begin{equation}
\label{eq:finite-matrix-recurrence}
   yP_k(y)=A_kP_{k+1}(y)+B_kP_k(y)+C_kP_{k-1}(y),
   \qquad 0\leq k\leq r-2.
\end{equation}
The last block row of \eqref{eq:finite-block-jacobi} instead gives the
terminal relation
\begin{equation}
\label{eq:finite-terminal-block-relation}
   \Theta(\nu_\ell)\mathbf p_{r-1}(\nu_\ell)
   =B_{r-1}\mathbf p_{r-1}(\nu_\ell)
   +C_{r-1}\mathbf p_{r-2}(\nu_\ell),
   \qquad 0\leq\ell\leq L-1.
\end{equation}

For the finite setting, write
\begin{equation}
\label{eq:finite-truncated-module}
   \mathcal M_m[x]_{<r}
   =\{P\in\mathcal M_m[x]:\deg P<r\}.
\end{equation}
For $P,Q\in\mathcal M_m[x]_{<r}$, define
\begin{equation}
\label{eq:finite-inner-product}
   \langle P,Q\rangle_\Theta
   =\sum_{\ell=0}^{L-1}
      P(\Theta(\nu_\ell))W_p(\nu_\ell)
      Q(\Theta(\nu_\ell))^\ast\omega_\ell.
\end{equation}
The following proposition gives the finite scalar-to-matrix identity and
proves that \eqref{eq:finite-inner-product} is non-degenerate on the
truncated module.

\begin{prop}
\label{prop:finite-scalar-to-matrix}
Under the hypotheses above, the finite family of matrix valued polynomials
$(P_k)_{0\leq k\leq r-1}$ satisfies
\begin{equation}
\label{eq:finite-matrix-to-scalar}
   P_k(\Theta(\nu_\ell))\mathbf p_0(\nu_\ell)
   =\mathbf p_k(\nu_\ell),
   \qquad
   0\leq k\leq r-1,\quad 0\leq\ell\leq L-1.
\end{equation}
For $0\leq j,k\leq r-1$, these polynomials satisfy
\begin{equation}
\label{eq:finite-matrix-orthogonality}
   \langle P_j,P_k\rangle_\Theta=\delta_{j,k}H_j,
   \qquad
   H_j=\operatorname{diag}
   (h_{mj},\ldots,h_{mj+m-1}).
\end{equation}
The corresponding discrete push-forward matrix measure is
\begin{equation}
   \label{eq:finite-pushforward-measure}
   \Sigma_\Theta
   =\sum_{\ell=0}^{L-1}
   \omega_\ell W_p(\nu_\ell)\,
   \delta_{\Theta(\nu_\ell)}.
\end{equation}
The resulting matrix inner product is non-degenerate on
$\mathcal M_m[x]_{<r}$. Moreover, $(P_k)_{0\leq k\leq r-1}$ is a basis of this
space as a left $\mathcal M_m$-module.
\end{prop}

\begin{proof}
Applying $\Theta$ to \eqref{eq:finite-J-eigenvectors} and reading the block
rows gives the vector recurrence associated with
\eqref{eq:finite-block-jacobi}. The induction used in
Proposition~\ref{prop:matrix-to-scalar}, now for $0\leq k\leq r-1$, proves
\eqref{eq:finite-matrix-to-scalar}; only the invertibility of
$A_0,\ldots,A_{r-2}$ is needed. Substituting
\eqref{eq:finite-matrix-to-scalar} into
\eqref{eq:finite-inner-product} and using
\eqref{eq:finite-scalar-orthogonality} proves
\eqref{eq:finite-matrix-orthogonality}.

We now prove non-degeneracy. Let $P\in\mathcal M_m[x]_{<r}$ and suppose that
$\langle P,P\rangle_\Theta=0$. By \eqref{eq:finite-inner-product},
\[
   \langle P,P\rangle_\Theta
   =\sum_{\ell=0}^{L-1}\omega_\ell
   v_\ell v_\ell^\ast,
   \qquad
   v_\ell=P(\Theta(\nu_\ell))\mathbf p_0(\nu_\ell).
\]
Every summand is positive semidefinite and every $\omega_\ell$ is positive.
Hence $v_\ell=0$ for $0\leq\ell\leq L-1$, that is,
\[
   P(\Theta(\nu_\ell))\mathbf p_0(\nu_\ell)=0,
   \qquad 0\leq\ell\leq L-1.
\]
Each entry of the vector polynomial
$P(\Theta(x))\mathbf p_0(x)$ has degree at most
$m(r-1)+m-1=L-1$. Since it vanishes at the $L$ distinct nodes
$\nu_0,\ldots,\nu_{L-1}$, it vanishes identically. The scalar polynomials
\[
   \{\Theta(x)^q p_s(x):0\leq q\leq r-1,\ 0\leq s\leq m-1\}
\]
have the distinct degrees $mq+s=0,\ldots,L-1$ and nonzero leading
coefficients. Hence they form a basis of the scalar polynomials of degree
less than $L$, and comparison row by row gives $P=0$.

Finally, the invertible leading coefficient of every $P_k$ shows that
the change of basis from $(I_m,xI_m,\ldots,x^{r-1}I_m)$ to
$(P_k)_{0\leq k\leq r-1}$ is triangular with invertible diagonal blocks.
Thus $(P_k)_{0\leq k\leq r-1}$ is a basis of
$\mathcal M_m[x]_{<r}$ as a left $\mathcal M_m$-module. This completes the
proof of the proposition.
\end{proof}

If several scalar nodes have the same image under $\Theta$, formula
\eqref{eq:finite-pushforward-measure} combines them. More
explicitly, for $y$ in the finite image of the support,
\begin{equation}
\label{eq:finite-collision-mass}
   \Sigma_\Theta(\{y\})
   =\sum_{\substack{0\leq\ell\leq L-1\\
                     \Theta(\nu_\ell)=y}}
   \omega_\ell W_p(\nu_\ell).
\end{equation}
The terminal block relation may equivalently be written
\[
   \bigl((yI_m-B_{r-1})P_{r-1}(y)
   -C_{r-1}P_{r-2}(y)\bigr)\big|_{y=\Theta(\nu_\ell)}
   \mathbf p_0(\nu_\ell)=0.
\]
It is a relation on the finite spectral support, not a recurrence defining an
additional polynomial $P_r$ with positive norm.

The divisibility assumption $m\mid L$ avoids a final partial block. When it
fails, $\Theta(J_L)$ still has finite bandwidth, but the last block has a
different size; that variant is outside the present scope. We also emphasize
that $J_L$ here is the complete Jacobi matrix of a genuinely finite scalar
measure. It should not be confused with a principal truncation of an
infinite Jacobi matrix: here $J_L$ is the complete finite operator associated
with the measure under consideration.

For the beefed-up Ehrenfest model introduced in \cite{grunbaum2009}, $L=N+1$
and
\[
   \nu_\ell=1-\frac{2\ell}{N},
   \qquad
   \omega_\ell=\binom{N}{\ell}2^{-N}.
\]
The polynomial map is quadratic, so $m=2$. The standing assumption that the
Ehrenfest population $N$ is odd gives $N+1=2r$, exactly the divisibility
condition \eqref{eq:finite-dimension-divisibility}. Hence, the family
constructed later in Subsection~\ref{sec:MVOP-Ehrenfestq} is a direct
specialization of Proposition~\ref{prop:finite-scalar-to-matrix}.

\section{Irreducibility of the matrix valued family}
\label{sec:irreducibility}

We now study whether the polynomials constructed above decompose, up to a
constant change of basis, into matrix valued orthogonal families of smaller
size. Since their orthogonality is described intrinsically by the push-forward
measure $\Sigma_\Theta$, reducibility must be formulated in the spectral
variable $y=\Theta(x)$.

Let $(P_n)_{n\geq0}$ be the matrix valued orthogonal polynomials constructed
in Section~\ref{sec:Construction:MVOP}. For an invertible matrix
$S\in\mathcal M_m$, define
\[
   d\Sigma_\Theta^S(y)=S\,d\Sigma_\Theta(y)\,S^\ast.
\]
Then $(P_nS^{-1})_{n\geq0}$ is orthogonal with respect to
$\Sigma_\Theta^S$. We say that $\Sigma_\Theta$, or equivalently
$(P_n)_{n\geq0}$, is \emph{reducible} if there exists such a matrix $S$ for
which $\Sigma_\Theta^S$ is block diagonal with at least two blocks of smaller
size. Otherwise, it is \emph{irreducible}.

Reducibility is characterized by the real vector space
\begin{equation}
\label{eq:commutant-pushforward}
   \mathcal A(\Sigma_\Theta)
   =
   \left\{T\in\mathcal M_m:
   T\Sigma_\Theta(E)=\Sigma_\Theta(E)T^\ast
   \text{ for every Borel set }E\right\}.
\end{equation}
The measure $\Sigma_\Theta$ is irreducible if and only if
$\mathcal A(\Sigma_\Theta)=\mathbb R I$, see
\cites{tiriao_zurrian_reducibility_2018,
koelink_roman_reducibility_2016}. If
$d\Sigma_\Theta(y)=\mathcal W(y)\,dy$, then
\eqref{eq:commutant-pushforward} is equivalent to
\begin{equation}
\label{eq:commutant-pushforward-density}
   T\mathcal W(y)=\mathcal W(y)T^\ast
   \qquad\text{for almost every }y.
\end{equation}
We say that $\mathcal A(\Sigma_\Theta)$ is $\ast$-invariant if
$T\in\mathcal A(\Sigma_\Theta)$ implies
$T^\ast\in\mathcal A(\Sigma_\Theta)$.

\subsection{Moment criteria for reducibility}

We begin with the moments of the push-forward measure. Set
\begin{equation}
\label{eq:pushforward-moments}
   M_r=\int_{\mathbb R}y^r\,d\Sigma_\Theta(y),
   \qquad r\in\mathbb N_0,
\end{equation}
and write $D=M_0$. By \eqref{eq:matrix-orthogonality},
\[
   D=H_0
   =\operatorname{diag}(h_0,\ldots,h_{m-1}).
\]

\begin{lemma}
\label{lem:zeroth-moment-structure}
If $T=(t_{ij})\in\mathcal A(\Sigma_\Theta)$, then
\begin{equation}
\label{eq:zeroth-moment-relation}
   TD=DT^\ast,
   \qquad\text{and}\qquad
   t_{ij}h_j=h_i\overline{t_{ji}}.
\end{equation}
Therefore, the diagonal entries of $T$ are real, and
\[
   U=D^{-1/2}TD^{1/2}
\]
is Hermitian. In particular, $T$ is diagonalizable and has real spectrum.
If, in addition, $T$ is Hermitian and $h_0,\ldots,h_{m-1}$ are pairwise
distinct,
then $T$ is a real diagonal matrix.
\end{lemma}

\begin{proof}
Since $P_0=I$, its norm matrix is the total mass of the pushed-forward
measure:
\[
   H_0=\int_{\mathbb R}d\Sigma_\Theta(y)=M_0=D.
\]
Integrating the defining relation
$T\,d\Sigma_\Theta=d\Sigma_\Theta\,T^\ast$ gives $TD=DT^\ast$.
Taking the $(i,j)$ entry yields the second identity in
\eqref{eq:zeroth-moment-relation}. Moreover,
\[
   U^\ast
   =D^{1/2}T^\ast D^{-1/2}
   =D^{-1/2}TD^{1/2}
   =U.
\]
If $T=T^\ast$, \eqref{eq:zeroth-moment-relation} reduces to $TD=DT$.
When the diagonal entries of $D$ are pairwise distinct, this commutation
relation forces all off-diagonal entries of $T$ to vanish.
\end{proof}

This gives a first structural restriction on the elements of
$\mathcal A(\Sigma_\Theta)$. Notice that distinct diagonal entries of
$H_0$ alone do not force $T$ to be diagonal: one first needs an additional
condition implying that $T$ is Hermitian. To obtain such information, we
now use the higher moments. These moments are intrinsic to the measure and
do not depend on the normalization chosen for the matrix polynomials. Define
the normalized moments
\begin{equation}
\label{eq:normalized-pushforward-moments}
   G_r=D^{-1/2}M_rD^{-1/2},
   \qquad r\geq1.
\end{equation}
The matrices $G_r$ are Hermitian, and they are real symmetric in the present
construction. For a finite set $I\subset\mathbb N$, let
\begin{equation}
\label{eq:normalized-moment-commutant}
   \mathcal C_I
   =
   \left\{X\in\mathcal M_m:
   XG_r=G_rX\ \text{for every }r\in I\right\}.
\end{equation}

\begin{prop}
\label{prop:moment-criterion}
Let $I\subset\mathbb N$ be finite. For every
$T\in\mathcal A(\Sigma_\Theta)$, the Hermitian matrix
\[
   U=D^{-1/2}TD^{1/2}
\]
belongs to $\mathcal C_I$. Therefore:
\begin{enumerate}[(i)]
   \item If $\mathcal C_I=\mathbb C I$, then
   $\mathcal A(\Sigma_\Theta)=\mathbb R I$, and the measure is irreducible.
   \item Suppose that every Hermitian $X\in\mathcal C_I$ commutes with $D$.
   Then every element of $\mathcal A(\Sigma_\Theta)$ is Hermitian, so
   $\mathcal A(\Sigma_\Theta)$ is $\ast$-invariant. If, in
   addition, the diagonal entries of $D$ are pairwise distinct, then every
   element of $\mathcal A(\Sigma_\Theta)$ is diagonal with real entries.
\end{enumerate}
\end{prop}

\begin{proof}
By Lemma~\ref{lem:zeroth-moment-structure}, $U$ is Hermitian. For
$r\geq1$, integration of \eqref{eq:commutant-pushforward} against $y^r$
gives
\[
   TM_r=M_rT^\ast,
   \qquad r\geq1.
\]
Using this relation and the identity
$U=U^\ast=D^{1/2}T^\ast D^{-1/2}$, we obtain
\begin{align*}
   UG_r
   &=D^{-1/2}TM_rD^{-1/2}\\
   &=D^{-1/2}M_rT^\ast D^{-1/2}\\
   &=G_rU.
\end{align*}
Thus $U$ is Hermitian and belongs to $\mathcal C_I$.

If $\mathcal C_I=\mathbb C I$, then $U$ is a real scalar matrix and hence so
is $T$, proving (i). Under the hypothesis in (ii), $UD=DU$, so $U$ commutes
with $D^{1/2}$ and
\[
   T=D^{1/2}UD^{-1/2}=U=U^\ast.
\]
The relation $TD=DT^\ast$ now becomes $TD=DT$. If the diagonal entries of $D$
are pairwise distinct, this forces the Hermitian matrix $T$ to be diagonal,
and its diagonal entries are real.
\end{proof}

The moments are computable from the block Jacobi matrix. If
$K=\Theta(J)$, then
\begin{equation}
\label{eq:moments-from-block-jacobi}
   M_r=(K^r)_{00}D.
\end{equation}
Indeed, the zeroth block row of
$K^r\mathbf p(x)=\Theta(x)^r\mathbf p(x)$ gives
\[
   \Theta(x)^r\mathbf p_0(x)
   =\sum_{j\geq0}(K^r)_{0j}\mathbf p_j(x).
\]
Multiplying from the right by $\mathbf p_0(x)^\ast$, integrating with
respect to $\mu$, and using \eqref{eq:matrix-orthogonality} proves
\eqref{eq:moments-from-block-jacobi}.
In particular,
\[
   M_1=B_0D,
   \qquad
   M_2=(B_0^2+A_0C_1)D.
\]
Thus Proposition~\ref{prop:moment-criterion} can often be checked using only
the first blocks of $K$.

\subsection{A support-graph criterion}

Once it has been established that every element of
$\mathcal A(\Sigma_\Theta)$ is real diagonal, the off-diagonal entries of
the push-forward density determine exactly which diagonal entries must
coincide. The following criterion makes this observation precise.

\begin{prop}
\label{prop:weight-support-graph}
Suppose that $\mathcal A(\Sigma_\Theta)$ consists of real diagonal matrices
and that $d\Sigma_\Theta(y)=\mathcal W(y)\,dy$. Define a graph
$G_{\mathcal W}$ with vertex set $\{0,\ldots,m-1\}$ by joining two distinct
vertices $i$ and $j$ whenever
\[
   \mathcal W_{ij}(y)\neq0
\]
on a set of positive Lebesgue measure. If $C_1,\ldots,C_s$ are the connected
components of $G_{\mathcal W}$, then
\begin{equation}
\label{eq:weight-support-graph-commutant}
\begin{aligned}
   \mathcal A(\Sigma_\Theta)
   =\bigl\{\operatorname{diag}(t_0,\ldots,t_{m-1}):{}
   &t_i\in\mathbb R,\\
   &t_i=t_j\ \text{whenever $i,j\in C_\ell$ for some $\ell$}\bigr\}.
\end{aligned}
\end{equation}
In particular, $\Sigma_\Theta$ is irreducible if and only if
$G_{\mathcal W}$ is connected.
\end{prop}

\begin{proof}
Since $\Sigma_\Theta$ is positive,
$\mathcal W(y)=\mathcal W(y)^\ast$ almost everywhere. Thus
$\mathcal W_{ji}(y)=\overline{\mathcal W_{ij}(y)}$ almost everywhere, and
$G_{\mathcal W}$ is undirected. Let $T\in\mathcal A(\Sigma_\Theta)$. By hypothesis,
$T=\operatorname{diag}(t_0,\ldots,t_{m-1})$ with $t_i\in\mathbb R$.
Taking the $(i,j)$ entry in
\eqref{eq:commutant-pushforward-density} gives
\begin{equation}
\label{eq:diagonal-weight-entry}
   (t_i-t_j)\mathcal W_{ij}(y)=0
   \qquad\text{for almost every }y.
\end{equation}
Hence $t_i=t_j$ along every edge of $G_{\mathcal W}$, and therefore on
each connected component.

Conversely, suppose that $T$ is a real diagonal matrix whose entries are
constant on the connected components of $G_{\mathcal W}$. If $i$ and $j$
belong to different components, then $\mathcal W_{ij}=0$ almost everywhere;
if they belong to the same component, then $t_i=t_j$. Thus
\eqref{eq:diagonal-weight-entry} holds for every $i,j$, which is equivalent
to $T\mathcal W=\mathcal WT^\ast$ almost everywhere. This proves
\eqref{eq:weight-support-graph-commutant}. The final assertion follows from
the fact that $\Sigma_\Theta$ is reducible if and only if
$\mathcal A(\Sigma_\Theta)=\mathbb R I$.
\end{proof}

The connected components describe the precise block decomposition detected
by the criterion. After a simultaneous permutation of rows and columns,
$\mathcal W$ is block diagonal, with one block for each connected component
of $G_{\mathcal W}$. A singleton component corresponds to an index $i$ for
which
\[
   \mathcal W_{ij}(y)=\mathcal W_{ji}(y)=0
   \qquad\text{for every }j\neq i,
\]
almost everywhere. For $m=2$, the graph is disconnected precisely when
$\mathcal W_{01}=\mathcal W_{10}=0$ almost everywhere.

For applications, it is enough to find a spanning tree whose corresponding
entries $\mathcal W_{ij}$ are nonzero on sets of positive measure. In
particular, if some row $k$ satisfies this condition for every
$j\neq k$, then $G_{\mathcal W}$ is connected and the measure is
irreducible.

\subsection{The $2\times2$ Hermite family}
\label{subsec:hermite-irreducibility}

In this subsection we construct a family of $2\times2$ Hermite-type matrix
valued orthogonal polynomials that is irreducible for $\alpha_1\neq0$ and
reducible for $\alpha_1=0$. Consider the Hermite polynomials $H_n$,
orthogonal on $\mathbb R$ with respect to $e^{-x^2}\,dx$, normalized so that
\[
   h_n=\int_{\mathbb R}H_n(x)^2e^{-x^2}\,dx
   =\sqrt{\pi}\,2^n n!.
\]
For $m=2$, the matrix density before push-forward is
\begin{equation}
\label{eq:hermite-source-weight}
   W_H(x)
   =
   e^{-x^2}
   \begin{pmatrix}
      1&2x\\
      2x&4x^2
   \end{pmatrix}.
\end{equation}
This matrix weight is irreducible in the original variable, before
push-forward.
Indeed, write
\[
   W_H(x)=e^{-x^2}v(x)v(x)^\ast,
   \qquad
   v(x)=\begin{pmatrix}1\\2x\end{pmatrix}.
\]
If a constant matrix $T$ satisfies $TW_H(x)=W_H(x)T^\ast$. Writing
$T=(t_{ij})_{i,j=0}^1$ and comparing powers of $x$ in this relation gives
$t_{01}=t_{10}=0$ and $t_{00}=t_{11}$. Substitution in
$TW_H=W_HT^\ast$ then shows that this common value is real. Hence
$T\in\mathbb R I$.

We now consider the variable $y=\Theta(x)$. Let
\begin{equation}
\label{eq:hermite-quadratic-theta}
   \Theta(x)=x^2+\alpha_1x,
   \qquad
   x_v=-\frac{\alpha_1}{2},
   \qquad
   y_v=\Theta(x_v)=-\frac{\alpha_1^2}{4}.
\end{equation}
For $y>y_v$, set
\begin{equation}
\label{eq:hermite-inverse-branches}
   d(y)=\sqrt{\alpha_1^2+4y}=2\sqrt{y-y_v},
   \qquad
   \phi_\pm(y)=x_v\pm\frac{d(y)}{2}.
\end{equation}
Then $\phi_\pm$ are the two inverse branches of $\Theta$, and
\[
   |\Theta'(\phi_\pm(y))|
   =
   |2\phi_\pm(y)+\alpha_1|
   =
   d(y).
\]
Therefore Proposition~\ref{prop:pushforward-measure} gives
\begin{equation}
\label{eq:hermite-pushforward-weight}
   \mathcal W_{\alpha_1}(y)
   =
   \frac{1}{d(y)}
   \sum_{\varepsilon\in\{+,-\}}
   e^{-\phi_\varepsilon(y)^2}
   \begin{pmatrix}
      1&2\phi_\varepsilon(y)\\
      2\phi_\varepsilon(y)&4\phi_\varepsilon(y)^2
   \end{pmatrix},
   \qquad y>y_v.
\end{equation}
In particular,
\begin{equation}
\label{eq:hermite-pushforward-offdiagonal}
   \bigl(\mathcal W_{\alpha_1}(y)\bigr)_{01}
   =
   \frac{2}{d(y)}
   \left(
   \phi_+(y)e^{-\phi_+(y)^2}
   +
   \phi_-(y)e^{-\phi_-(y)^2}
   \right).
\end{equation}

The following theorem shows that reducibility occurs precisely at the
symmetric value $\alpha_1=0$.

\begin{theorem}
\label{thm:hermite-irreducibility}
The $2\times2$ Hermite family associated with
$\Theta(x)=x^2+\alpha_1x$ is irreducible if $\alpha_1\neq0$ and reducible if
$\alpha_1=0$.
\end{theorem}

\begin{proof}
The total mass $D$ and the first two normalized moments defined by
\eqref{eq:normalized-pushforward-moments} are
\[
   D=\sqrt{\pi}
   \begin{pmatrix}1&0\\0&2\end{pmatrix},
   \qquad
   G_1=
   \begin{pmatrix}
      1/2&\alpha_1/\sqrt2\\
      \alpha_1/\sqrt2&3/2
   \end{pmatrix},
\]
and
\[
   G_2=
   \begin{pmatrix}
      \alpha_1^2/2+3/4&3\alpha_1/\sqrt2\\
      3\alpha_1/\sqrt2&3\alpha_1^2/2+15/4
   \end{pmatrix}.
\]
These formulas follow either by direct Gaussian integration or from
\eqref{eq:moments-from-block-jacobi}. The eigenvalues of $G_1$ are distinct
for every $\alpha_1$, since its discriminant is $1+2\alpha_1^2$, and
\begin{equation}
\label{eq:hermite-normalized-moment-commutator}
   [G_1,G_2]
   =
   \frac{\alpha_1^3}{\sqrt2}
   \begin{pmatrix}0&1\\-1&0\end{pmatrix}.
\end{equation}

Suppose first that $\alpha_1\neq0$. Since $G_1$ has simple spectrum, every
matrix commuting with it is of the form $aI+bG_1$, with $a,b\in\mathbb C$.
By
\eqref{eq:hermite-normalized-moment-commutator}, such a matrix commutes with
$G_2$ only if $b=0$. Hence
\[
   \mathcal C_{\{1,2\}}=\mathbb C I.
\]
Proposition~\ref{prop:moment-criterion}(i), with $I=\{1,2\}$, gives
$\mathcal A(\Sigma_\Theta)=\mathbb R I$, proving irreducibility.

Now suppose that $\alpha_1=0$. In this case
$\phi_\pm(y)=\pm\sqrt y$ and $d(y)=2\sqrt y$. The two off-diagonal
contributions in \eqref{eq:hermite-pushforward-weight} cancel, and
\begin{equation}
\label{eq:hermite-pushforward-diagonal}
   \mathcal W_0(y)
   =
   e^{-y}
   \begin{pmatrix}
      y^{-1/2}&0\\
      0&4y^{1/2}
   \end{pmatrix},
   \qquad y>0.
\end{equation}
The matrix weight is diagonal and therefore reducible. To identify
$\mathcal A(\Sigma_0)$ exactly, observe that
$G_1=\operatorname{diag}(1/2,3/2)$ has simple spectrum and commutes with
$D$. Every Hermitian matrix commuting with $G_1$ is diagonal and hence
commutes with $D$. Since the diagonal entries of $D$ are distinct,
Proposition~\ref{prop:moment-criterion}(ii) shows that every element of
$\mathcal A(\Sigma_0)$ is real diagonal. Moreover,
$G_{\mathcal W_0}$ consists of two isolated vertices. By
Proposition~\ref{prop:weight-support-graph}, the family is reducible and,
more precisely,
\[
   \mathcal A(\Sigma_0)
   =
   \left\{
   \begin{pmatrix}r_0&0\\0&r_1\end{pmatrix}:
   r_0,r_1\in\mathbb R
   \right\}.
\]
This completes the proof.
\end{proof}

At the reducible value $\alpha_1=0$, the decomposition can also be seen
directly in the polynomials. By the scalar-to-matrix identity
\eqref{eq:matrix-to-scalar}, we have
\[
   P_n(x^2)
   \begin{pmatrix}1\\2x\end{pmatrix}
   =
   \begin{pmatrix}H_{2n}(x)\\H_{2n+1}(x)\end{pmatrix}.
\]
We make this relation explicit using parity. Write
\[
   P_n(y)=
   \begin{pmatrix}A_n(y)&B_n(y)\\C_n(y)&D_n(y)\end{pmatrix}.
\]
Evaluating the preceding identity at $x$ and at $-x$, and using
$H_{2n}(-x)=H_{2n}(x)$ and
$H_{2n+1}(-x)=-H_{2n+1}(x)$, gives
\[
\begin{aligned}
   A_n(x^2)\pm2xB_n(x^2)&=H_{2n}(x),\\
   C_n(x^2)\pm2xD_n(x^2)&=\pm H_{2n+1}(x),
\end{aligned}
\]
where the two choices of signs correspond to $x$ and $-x$, respectively.
Solving these two systems we get
$B_n=C_n=0$, $A_n(x^2)=H_{2n}(x)$, and
$D_n(x^2)=H_{2n+1}(x)/(2x)$. Therefore,
\[
   P_n(y)
   =
   \begin{pmatrix}
      H_{2n}(\sqrt y)&0\\
      0&H_{2n+1}(\sqrt y)/(2\sqrt y)
   \end{pmatrix}.
\]
Both diagonal entries are polynomials in $y$: the first because $H_{2n}$ is
even, and the second because the odd polynomial $H_{2n+1}(x)$ is divisible
by $x$.
Using the classical Hermite--Laguerre identities \cites{Koekoek,DLMF},
\[
   H_{2n}(x)
   =
   (-1)^n2^{2n}n!L_n^{(-1/2)}(x^2),
   \qquad
   H_{2n+1}(x)
   =
   (-1)^n2^{2n+1}n!xL_n^{(1/2)}(x^2),
\]
we obtain
\begin{equation}
\label{eq:hermite-reducible-polynomials}
   P_n(y)
   =
   (-1)^n2^{2n}n!
   \begin{pmatrix}
      L_n^{(-1/2)}(y)&0\\
      0&L_n^{(1/2)}(y)
   \end{pmatrix}.
\end{equation}
This is compatible with the two scalar components in
\eqref{eq:hermite-pushforward-diagonal}, which are precisely Laguerre weights with
parameters $-1/2$ and $1/2$, up to the constant factor $4$ in the second
component.

\begin{remark}
This example shows that irreducibility is not visible from the source density
alone. The source matrix weight $W_H$ is irreducible. When
$\Theta(x)=x^2$, the two points $x$ and $-x$ in each fiber have the same
Hermite scalar weight, so their off-diagonal contributions cancel after
push-forward. The even and odd Hermite polynomials then separate into two
scalar Laguerre families.

For $\alpha_1\neq0$, the involution
$x\mapsto-x-\alpha_1$ still exchanges the two points in every fiber of
$\Theta$, but it no longer preserves the Hermite weight $e^{-x^2}$. The
branch contributions do not cancel, and the family is irreducible. Thus
reducibility may occur at an isolated symmetric parameter value even though
the scalar family, matrix size, and degree of $\Theta$ remain fixed.
\end{remark}

\subsection{A cubic $3\times3$ example with variable pointwise rank}
\label{subsec:cubic-variable-rank}
This example separates two phenomena. The pointwise rank of the push-forward
density is a local property. The number of active inverse branches gives an
upper bound for this rank, and in the present example equality follows from
a Vandermonde argument. Reducibility, by contrast, is a global property: it
asks whether one constant change of basis splits the matrix measure for
almost every value of the spectral variable.

We construct a $3\times3$ family whose push-forward density has rank one,
two, and three on nonempty intervals, while the matrix valued family remains
irreducible.

The first ingredient is the cubic map
\begin{equation}
\label{eq:cubic-theta}
   \Theta(x)=x^3-3x,
   \qquad
   \Theta'(x)=3(x^2-1),
\end{equation}
which has three monotonicity intervals separated by the critical points $-1$
and $1$. We choose the asymmetric scalar support
\[
   I=[-3/2,3]
\]
so that one, two, and three branches are active on nonempty intervals of the
push-forward variable. The critical values of $\Theta$ are $2$ and $-2$.
The images of the endpoints of $I$ are
\[
   \Theta(-3/2)=\frac98,
   \qquad
   \Theta(3)=18.
\]
The three monotonicity intervals in the interior of the scalar support and
their images are
\[
   (-3/2,-1)\longmapsto(9/8,2),
   \qquad
   (-1,1)\longmapsto(-2,2),
   \qquad
   (1,3)\longmapsto(-2,18).
\]
If $k(y)$ is the number of distinct solutions of
$\Theta(x)=y$ in the interior of $I$, the active branches and their number
are as follows:
\begin{center}
\renewcommand{\arraystretch}{1.25}
\begin{tabular}{@{}c l c@{}}
\toprule
\textbf{Range} & \textbf{Active intervals in $x$}
   & \textbf{Branch count}\\
\midrule
$(-2,9/8)$ & $(-1,1),\ (1,3)$ & $k(y)=2$\\
$(9/8,2)$ & $(-3/2,-1),\ (-1,1),\ (1,3)$ & $k(y)=3$\\
$(2,18)$ & $(1,3)$ & $k(y)=1$\\
\bottomrule
\end{tabular}
\end{center}
Equivalently, away from the critical values and endpoint images,
\begin{equation}
\label{eq:cubic-branch-count}
   k(y)=
   \begin{cases}
      2,&-2<y<9/8,\\
      3,&9/8<y<2,\\
      1,&2<y<18.
   \end{cases}
\end{equation}

Figure~\ref{fig:cubic-pushforward-geometry} displays the polynomial $\Theta$.
The three colored pieces of the cubic are its monotonicity intervals inside
the scalar support. The colored portions of the vertical axis indicate the
three regions in the image and the corresponding number of active inverse
branches.

The second ingredient is a rescaled Legendre family on $I$. Let $\ell_n$ be
the ordinary Legendre polynomial and introduce the affine coordinate
$t\in[-1,1]$ by
\begin{equation}
\label{eq:cubic-transported-legendre}
   t=\frac{4x-3}{9},
   \qquad x=\frac{9t+3}{4},
   \qquad
   p_n(x)=\ell_n(t),
   \qquad
   d\mu(x)=\frac49\,dx.
\end{equation}
The coordinate $t$ transports the Legendre family from $[-1,1]$ to
$I=[-3/2,3]$, so that the critical points $\pm1$ of $\Theta$ lie in the
interior of $I$. The polynomials $(p_n)_{n\geq0}$ have the usual Legendre
norms
\begin{equation}
\label{eq:cubic-transported-legendre-norms}
   h_n=\frac{2}{2n+1}.
\end{equation}
Moreover
\[
   \mathbf p_0(x)
   =
   \begin{pmatrix}
      1\\[2pt]
      t\\[2pt]
      (3t^2-1)/2
   \end{pmatrix}.
\]

\begin{figure}[t]
  \centering
  \begin{tikzpicture}
    \begin{axis}[
        axis lines=middle,
        xmin=-2.1, xmax=3.4,
        ymin=-3.5, ymax=20,
        xlabel={$x$}, ylabel={$y$},
        domain=-2:3.1,
        samples=300,
        width=0.9\textwidth,
        height=9cm,
        enlargelimits=false,
        clip=false,
        xtick=\empty,
        ytick=\empty
      ]
      \addplot[thick, color=black!45] {x^3-3*x};
      \addplot[ultra thick, color=magenta!70!blue, domain=-1.5:-1,
        samples=80] {x^3-3*x};
      \addplot[ultra thick, color=blue!70, domain=-1:1,
        samples=120] {x^3-3*x};
      \addplot[ultra thick, color=green!55!black, domain=1:3,
        samples=140] {x^3-3*x};

      \addplot[ultra thick, color=magenta!70!blue]
        coordinates {(-1.5,0) (-1,0)};
      \addplot[ultra thick, color=blue!70]
        coordinates {(-1,0) (1,0)};
      \addplot[ultra thick, color=green!55!black]
        coordinates {(1,0) (3,0)};
      \addplot[only marks, mark=*]
        coordinates {(-1.5,0) (-1,0) (1,0) (3,0)};
      \node[below left] at (axis cs:-1.5,0) {$-\frac32$};
      \node[below] at (axis cs:-1,0) {$-1$};
      \node[above right] at (axis cs:1,0) {$1$};
      \node[below right] at (axis cs:3,0) {$3$};

      \addplot[dotted, thick] coordinates {(-1.5,0) (-1.5,1.125)};
      \addplot[dotted, thick] coordinates {(-1,0) (-1,2)};
      \addplot[dotted, thick] coordinates {(1,0) (1,-2)};
      \addplot[dotted, thick] coordinates {(3,0) (3,18)};
      \addplot[dotted, thick] coordinates {(-1.5,1.125) (0,1.125)};
      \addplot[dotted, thick] coordinates {(-1,2) (0,2)};
      \addplot[dotted, thick] coordinates {(0,-2) (1,-2)};
      \addplot[dotted, thick] coordinates {(0,18) (3,18)};

      \addplot[ultra thick, color=cyan!70!black]
        coordinates {(0,-2) (0,1.125)};
      \addplot[ultra thick, color=orange!85!black]
        coordinates {(0,1.125) (0,2)};
      \addplot[ultra thick, color=cyan!35!black, opacity=0.9]
        coordinates {(0,2) (0,18)};
      \addplot[only marks, mark=*]
        coordinates {(0,-2) (0,1.125) (0,2) (0,18)};
      \node[below right] at (axis cs:0,-2) {$-2$};
      \node[right, yshift=-5pt] at (axis cs:0,1.225) {$\frac98$};
      \node[right, yshift=5pt] at (axis cs:0,2) {$2$};
      \node[below right] at (axis cs:0,18) {$18$};
      \node[left] at (axis cs:0,-1.35) {$k(y)=2$};
      \node[anchor=east] (threebranches) at (axis cs:1.3,2.3)
        {$k(y)=3$};
      \draw[->, thin] (threebranches.west) -- (axis cs:0.03,1.5625);
      \node[left] at (axis cs:0,10) {$k(y)=1$};
    \end{axis}
  \end{tikzpicture}
  \caption{Scheme of the cubic push-forward on the interval
  $I=[-3/2,3]$. The critical points $-1$ and $1$ separate the three
  monotonicity intervals of $\Theta(x)=x^3-3x$. In the spectral variable,
  two inverse branches are active for $-2<y<9/8$, all three are active for
  $9/8<y<2$, and only one remains for $2<y<18$.}
  \label{fig:cubic-pushforward-geometry}
\end{figure}

The following theorem shows that all three pointwise ranks occur, while the
associated matrix valued family remains irreducible.

\begin{theorem}
\label{thm:cubic-rank-irreducibility}
For the transported Legendre family in
\eqref{eq:cubic-transported-legendre} and the cubic map
\eqref{eq:cubic-theta}, the push-forward density has pointwise rank $k(y)$
for almost every $y$. In particular, it has rank two, rank three, and rank
one on the respective intervals in \eqref{eq:cubic-branch-count}. The
associated $3\times3$ matrix valued family is irreducible:
\[
   \mathcal A(\Sigma_\Theta)=\mathbb R I_3.
\]
\end{theorem}

\begin{proof}
We first determine the pointwise rank. Let
$\phi_1(y),\ldots,\phi_{k(y)}(y)$ be the active inverse branches. At every
regular value under consideration,
Proposition~\ref{prop:pushforward-measure} gives
\begin{equation}
\label{eq:cubic-pushforward-density}
   \mathcal W(y)
   =
   \sum_{j=1}^{k(y)}
   \frac{4}{9|\Theta'(\phi_j(y))|}
   \mathbf p_0(\phi_j(y))\mathbf p_0(\phi_j(y))^\ast
   =V(y)\Delta(y)V(y)^\ast,
\end{equation}
where the columns of $V(y)$ are the vectors
$\mathbf p_0(\phi_j(y))$ and
\[
   \Delta(y)=\operatorname{diag}\left(
      \frac{4}{9|\Theta'(\phi_1(y))|},\ldots,
      \frac{4}{9|\Theta'(\phi_{k(y)}(y))|}
   \right)
\]
is positive definite. To determine the rank of $V(y)$, set
\[
   t_j(y)=\frac{4\phi_j(y)-3}{9}.
\]
Since
\[
   \mathbf p_0(x)
   =
   L\begin{pmatrix}1\\t\\t^2\end{pmatrix},
   \qquad
   L=
   \begin{pmatrix}
      1&0&0\\
      0&1&0\\
      -1/2&0&3/2
   \end{pmatrix},
   \qquad \det L=\frac32\neq0,
\]
we have the factorization
\[
   V(y)=L
   \begin{pmatrix}
      1&\cdots&1\\
      t_1(y)&\cdots&t_{k(y)}(y)\\
      t_1(y)^2&\cdots&t_{k(y)}(y)^2
   \end{pmatrix}.
\]
The second factor is a rectangular Vandermonde matrix. The inverse images
$\phi_j(y)$ are distinct, and the affine map $x\mapsto t$ is injective, so
the numbers $t_j(y)$ are distinct. For $k(y)=3$, its determinant is
\[
   \prod_{1\leq i<j\leq3}\bigl(t_j(y)-t_i(y)\bigr)\neq0.
\]
For $k(y)=2$, the $2\times2$ minor formed by the first two rows has
determinant $t_2(y)-t_1(y)\neq0$, while the case $k(y)=1$ is immediate from
the first entry. Thus, in every case, the $k(y)$ columns of $V(y)$ are
linearly independent. Finally, since $\Delta(y)$ is positive definite,
\[
   V(y)\Delta(y)V(y)^\ast
   =\bigl(V(y)\Delta(y)^{1/2}\bigr)
    \bigl(V(y)\Delta(y)^{1/2}\bigr)^\ast,
\]
and hence
\[
   \operatorname{rank}\mathcal W(y)
   =\operatorname{rank}V(y)
   =k(y),
\]
which proves the three rank assertions. The exceptional set
$\{-2,9/8,2,18\}$ has measure zero and does not affect this
almost-everywhere statement.

It remains to prove irreducibility. On $2<y<18$, there is a unique active
branch $\phi(y)\in(2,3)$; set
$t(y)=(4\phi(y)-3)/9\in(5/9,1)$. Equation
\eqref{eq:cubic-pushforward-density} becomes
\[
   \mathcal W(y)
   =c(y)v(t(y))v(t(y))^\ast,
   \qquad
   c(y)=\frac{4}{9|\Theta'(\phi(y))|}>0.
\]
Here
\[
   v(t)=\begin{pmatrix}1\\t\\(3t^2-1)/2\end{pmatrix}.
\]
Let $T\in\mathcal A(\Sigma_\Theta)$. Since $v(t)v(t)^\ast$ is a simple matrix valued polynomial of degree four in $t$, by examining each entry of the relation $Tv(t)v(t)^\ast = v(t)v(t)^\ast$ we obtain that $T$ is a real scalar multiple of the identity. Therefore,
$\mathcal A(\Sigma_\Theta)=\mathbb R I_3$, proving irreducibility.
\end{proof}

Figure~\ref{fig:cubic-pushforward-density} illustrates the size and the
spectral profile of the density in \eqref{eq:cubic-pushforward-density}. If
\(\lambda_1(y)\geq\lambda_2(y)\geq\lambda_3(y)\geq0\) are the eigenvalues
of \(\mathcal W(y)\), define
\[
   \rho_i(y)=\frac{\lambda_i(y)}{\operatorname{tr}\mathcal W(y)},
   \qquad 1\leq i\leq3.
\]
Thus, the number of positive normalized eigenvalues agrees with the
pointwise rank of \(\mathcal W(y)\).

\begin{figure}[t]
  \centering
  \begin{minipage}[t]{0.49\textwidth}
    \centering
    \includegraphics[width=\textwidth]{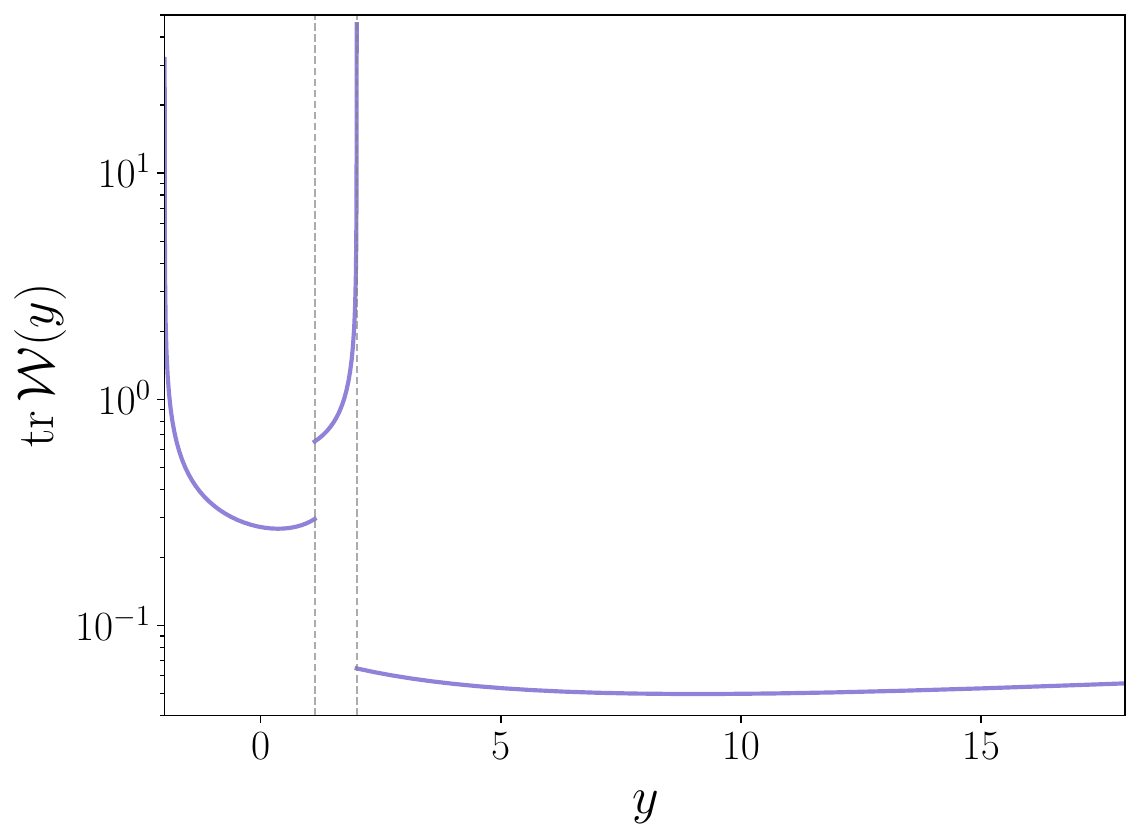}
  \end{minipage}\hfill
  \begin{minipage}[t]{0.49\textwidth}
    \centering
    \includegraphics[width=\textwidth]{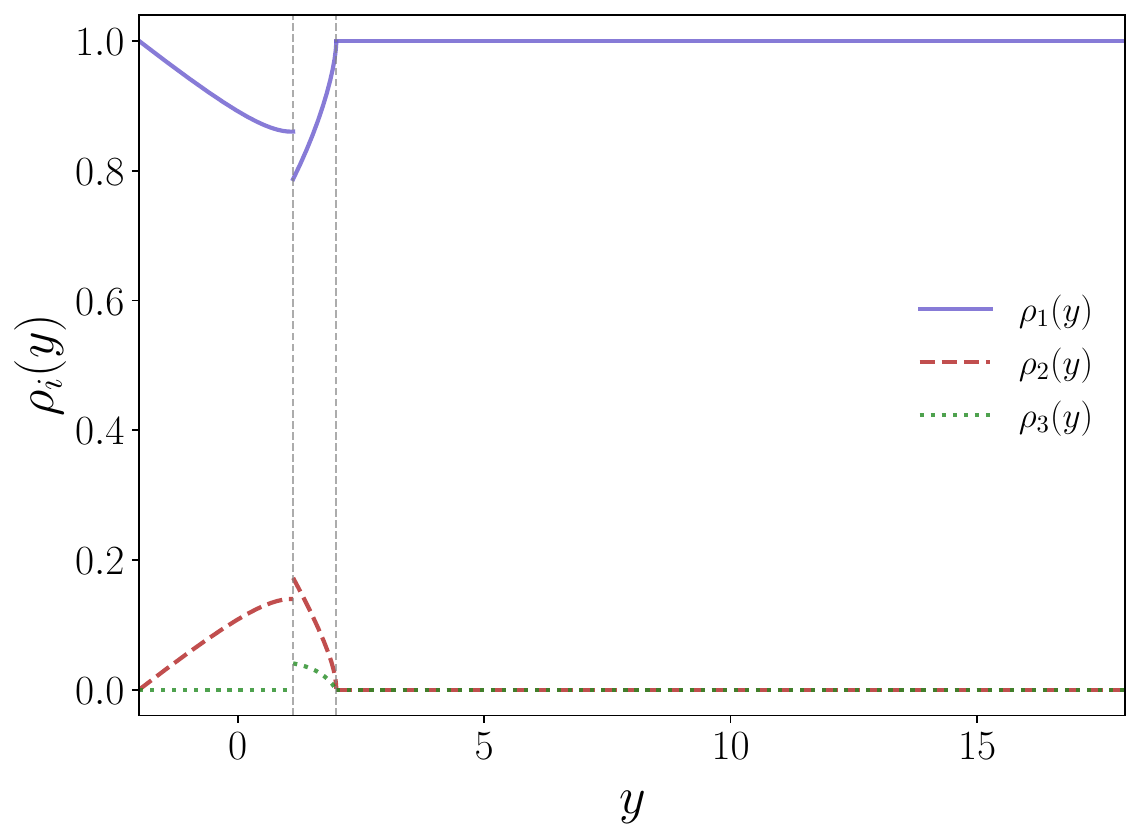}
  \end{minipage}
  \caption{The push-forward density in the cubic example. Left: the trace
  of \(\mathcal W(y)\), shown on a logarithmic vertical scale. Right: the
  normalized eigenvalues \(\rho_1(y)\geq\rho_2(y)\geq\rho_3(y)\). The
  dashed lines mark the endpoint image \(9/8\) and the critical value \(2\).
  The normalized spectrum displays two, three, and one positive eigenvalues
  on the respective intervals in \eqref{eq:cubic-branch-count}.}
  \label{fig:cubic-pushforward-density}
\end{figure}

\section{The beefed-up Ehrenfest model: motivating problem and explicit solution}
\label{sec:Krawtchouk-MV}

We now return to the random-walk model that motivated the construction.
Consider $N$ distinguishable balls distributed between two urns, with the
state recording the number of balls in the first urn. In the classical
Ehrenfest dynamics, one ball is chosen uniformly at each step and moved to the
opposite urn, so the state changes by one. The modified dynamics considered
here combines this rule with a two-ball move: with a prescribed probability,
two distinct balls are chosen and both are moved to the opposite urn. Thus the
state may also change by two or remain fixed, and the resulting transition
matrix is five-diagonal.

We first recall the scalar Krawtchouk spectral description of the classical
chain and then define the mixed one- and two-ball model precisely. We prove
that its transition matrix is a quadratic polynomial in the transition matrix
of the classical chain, which transfers the Krawtchouk spectral data to the
modified dynamics. For an odd number of balls, grouping consecutive states in
pairs turns the five-diagonal matrix into a block tridiagonal matrix and
produces a finite family of $2\times2$ matrix valued orthogonal polynomials and
the corresponding matrix measure of orthogonality.

\subsection{The classical Ehrenfest chain and its Krawtchouk spectrum}
The classical Ehrenfest urn model \cite{ehrenfest1907} consists of two urns,
labeled $A$ and $B$, containing $N$ distinguishable balls. Its state space is
$\{0,1,\ldots,N\}$, and
\[
   X_n=\#\{\text{balls in urn $A$ after $n$ moves}\},
   \qquad n\in\mathbb N_0.
\]
At each time step, one ball is chosen uniformly at random and moved to the
other urn. Therefore,
\[
   \mathbb P\bigl(X_{n+1}=j\mid X_n=i\bigr)
   =
   \begin{cases}
      \displaystyle\frac{i}{N},&j=i-1,\\[1ex]
      \displaystyle\frac{N-i}{N},&j=i+1,\\[1ex]
      0,&\text{otherwise}.
   \end{cases}
\]
The transition matrix $M_0$ has the tridiagonal form
\begin{equation}
\label{eq:M_0}
   M_0
   =\sum_{i=1}^N\frac{i}{N}E_{i,i-1}
    +\sum_{i=0}^{N-1}\frac{N-i}{N}E_{i,i+1},
\end{equation}
where $E_{i,j}$ denotes the matrix with a $1$ in the $(i,j)$ entry and zeros
elsewhere.

The spectral decomposition of $M_0$ is given by the Krawtchouk polynomials
\cite{karlin_mcgregor_ehrenfest_1965}. With the normalization used here, they
satisfy
\begin{equation}
\label{eq:recurrence_Krawtchouk}
-xK_n(x)
=\frac{N-n}{2}K_{n+1}(x)-\frac{N}{2}K_n(x)
 +\frac{n}{2}K_{n-1}(x),
 \qquad 0\leq n\leq N,
\end{equation}
with $K_{-1}=0$ and $K_0=1$; see, for instance, \cites{Koekoek,DLMF}.
Evaluating \eqref{eq:recurrence_Krawtchouk} at $x=j$ gives
\begin{equation}
\label{eq:lambda_j_Krawtchouk}
   M_0v_j=\lambda_jv_j,
   \qquad
   v_j=\begin{pmatrix}K_0(j)&\cdots&K_N(j)\end{pmatrix}^{\mathsf T},
   \qquad
   \lambda_j=1-\frac{2j}{N},
\end{equation}
for $0\leq j\leq N$. Moreover,
\begin{equation}
\label{eq:Krawtchouk-orthogonality}
   \sum_{x=0}^N\binom{N}{x}2^{-N}K_m(x)K_\ell(x)
   =\delta_{m,\ell}\binom{N}{m}^{-1}
   =2^{-N}\pi_m^{-1}\delta_{m,\ell},
   \qquad
   \pi_m=\binom{N}{m}2^{-N}.
\end{equation}
The binomial distribution $(\pi_j)_{0\leq j\leq N}$ is stationary for
$M_0$. The eigenvalues are $1-2j/N$, $0\leq j\leq N$. We use
this spectral description as the scalar input for the mixed model below.

\subsection{The mixed one- and two-ball model}
\label{sec:Ehrefest-with-q}
Assume $N\geq2$ and fix $q\in[0,1]$. Starting from a configuration of $N$
distinguishable balls in the two urns, retain the state variable $X_n$ defined
above and choose one
of the following moves independently at each time step:
\begin{enumerate}
   \item With probability $q$, choose two distinct balls uniformly at random
      and move both to the opposite urn.
   \item With probability $1-q$, choose one ball uniformly at random and move
      it to the opposite urn, as in the classical Ehrenfest chain.
\end{enumerate}
Let $J_2$ denote the transition matrix of the pure two-ball dynamics and set
\begin{equation}
\label{eq:Mq-mixture}
   M_q=(1-q)M_0+qJ_2.
\end{equation}
Thus $M_q$ is the transition matrix of the mixed chain. From state $i$, a
one-ball move leads to $i-1$ or $i+1$, whereas a two-ball move leads to
$i-2$, $i$, or $i+2$. Its nonzero transition probabilities are therefore
\begin{equation}
\label{eq:transition_prob_Mq}
\begin{aligned}
(M_q)_{i,i-2}&=q\frac{i(i-1)}{N(N-1)},
& (M_q)_{i,i-1}&=(1-q)\frac{i}{N},\\
(M_q)_{i,i}&=q\frac{2i(N-i)}{N(N-1)},
& (M_q)_{i,i+1}&=(1-q)\frac{N-i}{N},\\
(M_q)_{i,i+2}&=q\frac{(N-i)(N-i-1)}{N(N-1)}.
\end{aligned}
\end{equation}
Here, entries with indices outside $\{0,\ldots,N\}$ are absent. In the
two-ball move, the state remains at $i$ precisely when one ball is chosen from
each urn. This accounts for the diagonal probability
$i(N-i)/\binom{N}{2}=2i(N-i)/(N(N-1))$.
The choice $q=0$ recovers the classical chain, $q=1$ gives the pure two-ball
chain, and $q=1/2$ is the model considered in \cite{grunbaum2009}.

The next result identifies the mixed dynamics as a polynomial transform of
the classical chain.
\begin{prop}
   \label{prop:Mq=M02}
   The transition matrix of the mixed chain satisfies
   \begin{equation}
   \label{eq:Mq-polynomial}
      M_q=\Theta(M_0)
      =\frac{qN}{N-1}M_0^2+(1-q)M_0-\frac{q}{N-1}I,
   \end{equation}
   where
   \[
      \Theta(x)=\frac{qN}{N-1}x^2+(1-q)x-\frac{q}{N-1}.
   \]
   Moreover, the eigenvalues of $M_q$ are $\Theta(\lambda_j)$, where the
   $\lambda_j$ are the eigenvalues of $M_0$ in
   \eqref{eq:lambda_j_Krawtchouk}. Explicitly, if
   $\mu_j(q)=\Theta(\lambda_j)$, then
   \begin{equation}
      \label{eq:Theta-lambdaj-EhrenfestG}
      \mu_j(q)=\frac{N-2j}{N}
      -\frac{2qj(N-2j+1)}{N(N-1)},
      \qquad j=0,\ldots,N.
   \end{equation}
\end{prop}
\begin{proof}
Two successive one-ball moves select the same ball with probability $1/N$. In this case, the second move reverses the first. Indeed,
the probability to choose the ball labeled $m$ in the first and second move is $1/N^2$ which
summed over all the balls gives $1/N$.  On the other hand, the probability to select different balls in the first and second move is
$(N-1)/N^2$. Therefore:
\[
   M_0^2=\frac1N I+\frac{N-1}{N}J_2.
\]
Solving for $J_2$ and substituting the result into \eqref{eq:Mq-mixture}
gives \eqref{eq:Mq-polynomial}. Since the vectors $v_j$ in
\eqref{eq:lambda_j_Krawtchouk} form an eigenbasis for $M_0$, they are also
eigenvectors of $M_q=\Theta(M_0)$, with eigenvalues
$\Theta(\lambda_j)$. Substituting $\lambda_j=1-2j/N$ yields
\eqref{eq:Theta-lambdaj-EhrenfestG}.
\end{proof}

We now assume that $N$ is odd and group the states into the consecutive pairs
\[
   \{0,1\},\{2,3\},\ldots,\{N-1,N\}.
\]
Because the chain makes jumps of size at most two, transitions between these
pairs occur only within the same pair or between neighboring pairs. With
respect to this grouping, $M_q$ is the block tridiagonal matrix
\begin{equation}
\label{eq:Mq_Ehrenfestmatricial}
   M_q=
   \begin{pmatrix}
      B_{0} & A_{0}  &                   &                   &                   \\
      C_{1} & B_{1}  & A_{1}             &                   &                   \\
            & \ddots & \ddots            & \ddots            &                   \\
            &        & C_{\frac{N-3}{2}} & B_{\frac{N-3}{2}} & A_{\frac{N-3}{2}} \\
            &        &                   & C_{\frac{N-1}{2}} & B_{\frac{N-1}{2}}
   \end{pmatrix},
\end{equation}
where the rows and columns of block $i$ correspond to the states
$\{2i,2i+1\}$. The transition probabilities in
\eqref{eq:transition_prob_Mq} give
\begin{align*}
   A_i&=
   \begin{pmatrix}
      \frac{q(N-2i)(N-2i-1)}{N(N-1)} & 0                                \\
      \frac{(1-q)(N-2i-1)}{N}        & \frac{q(N-2i-1)(N-2i-2)}{N(N-1)}
   \end{pmatrix}, \\
   B_i&=
   \begin{pmatrix}
      \frac{4qi(N-2i)}{N(N-1)} & \frac{(1-q)(N-2i)}{N}           \\
      \frac{(1-q)(2i+1)}{N}    & \frac{2q(2i+1)(N-2i-1)}{N(N-1)}
   \end{pmatrix}, \qquad
   C_i=
   \begin{pmatrix}
      \frac{2qi(2i-1)}{N(N-1)} & \frac{2i(1-q)}{N}        \\
      0                         & \frac{2qi(2i+1)}{N(N-1)}
   \end{pmatrix}.
\end{align*}
For $q>0$, every upper off-diagonal block $A_i$ in
\eqref{eq:Mq_Ehrenfestmatricial} is invertible. The block recurrence therefore
produces a genuine finite family of $2\times2$ matrix valued orthogonal
polynomials. When $q=0$, the quadratic term in $\Theta$ disappears, the
blocks $A_i$ are singular, and the construction reduces to the scalar
Ehrenfest problem governed by $M_0$.

\subsection{Eigenvalue multiplicities and spectral collisions}

The case $q=0$ is the classical Ehrenfest model, while $q=1/2$ is the model
considered in \cite{grunbaum2009}. By Proposition~\ref{prop:Mq=M02}, the
Krawtchouk eigenvectors $v_j$ remain eigenvectors for every value of $q$, but
their eigenvalues may cease to be distinct. We say that a spectral collision
occurs when two distinct Krawtchouk eigenvectors $v_j$ and $v_k$ correspond to
the same eigenvalue of $M_q$:
\[
   \Theta(\lambda_j)=\Theta(\lambda_k).
\]
Since $M_q$ retains the Krawtchouk eigenbasis, such a collision produces an
eigenvalue of multiplicity two without causing a loss of diagonalizability.
The problem is therefore reduced to determining when the quadratic map
$\Theta$ identifies two points in the finite spectrum of $M_0$.

The following proposition provides a proof of the spectral pattern stated in
\cite{grunbaum2009}*{p.~273} and extends it to a general parameter $q$.
\begin{prop}
\label{prop:ehrenfest-spectral-collisions}
Let $\Omega$ be the set
\[
   \Omega=\left\{\left(3-\frac{2i}{N-1}\right)^{-1}
   :i=0,\ldots,N-1\right\}.
\]
If $q\notin\Omega$, then $M_q$ has simple spectrum. If
\[
   q=\left(3-\frac{2i}{N-1}\right)^{-1},
   \qquad i\in\{0,\ldots,N-1\},
\]
then $M_q$ has exactly $\lfloor i/2\rfloor+1$ double eigenvalues and all its
remaining eigenvalues are simple. Equivalently, the number of double
eigenvalues is $i/2+1$ when $i$ is even and $(i+1)/2$ when $i$ is odd.
\end{prop}
\begin{proof}
For $q>0$, the polynomial $\Theta$ is quadratic, so no value can have more
than two distinct preimages. For $0\leq j<k\leq N$, factoring the difference
$\Theta(\lambda_j)-\Theta(\lambda_k)$ gives
\[
 \Theta(\lambda_j)=\Theta(\lambda_k)
 \quad\Longleftrightarrow\quad
 q=\frac{N-1}{2(j+k)-N-1}.
\]
Writing $S=j+k$, the expression on the right belongs to $[0,1]$ precisely
when $N\leq S\leq2N-1$. Hence there is a unique
$i\in\{0,\ldots,N-1\}$ such that $S=2N-1-i$, and then
\[
   q=\left(3-\frac{2i}{N-1}\right)^{-1}.
\]
It remains to count the pairs with this sum. The change of variables
\[
   a=N-k,\qquad b=N-j
\]
gives a bijection between pairs $0\leq j<k\leq N$ satisfying
$j+k=2N-1-i$ and pairs $0\leq a<b$ satisfying $a+b=i+1$. There are
$\lfloor i/2\rfloor+1$ such pairs. Each pair gives one double eigenvalue,
and distinct pairs cannot give the same value because $\Theta$ has at most
two preimages. For $q=0$, the map $\Theta(x)=x$ is injective on the spectrum
of $M_0$. This proves the result.
\end{proof}
\begin{remark}
The smallest parameter in $\Omega$ is $1/3$. Thus there are no spectral
collisions for $0\leq q<1/3$, while at $q=1/3$ the first double eigenvalue
appears. At the other endpoint, $q=1$, the dynamics moves two balls at every
step and therefore preserves the parity of the state. Under the standing
assumption that $N$ is odd, the even and odd closed classes have the same
spectrum: $M_1$ has $(N+1)/2$ distinct eigenvalues, each of multiplicity two.
Figure~\ref{fig:eigenvalues} displays how these collisions arise as $q$
varies.
\end{remark}

\begin{figure}[t]
  \centering
  \begin{minipage}[t]{0.49\textwidth}
    \centering
    \includegraphics[width=\textwidth]{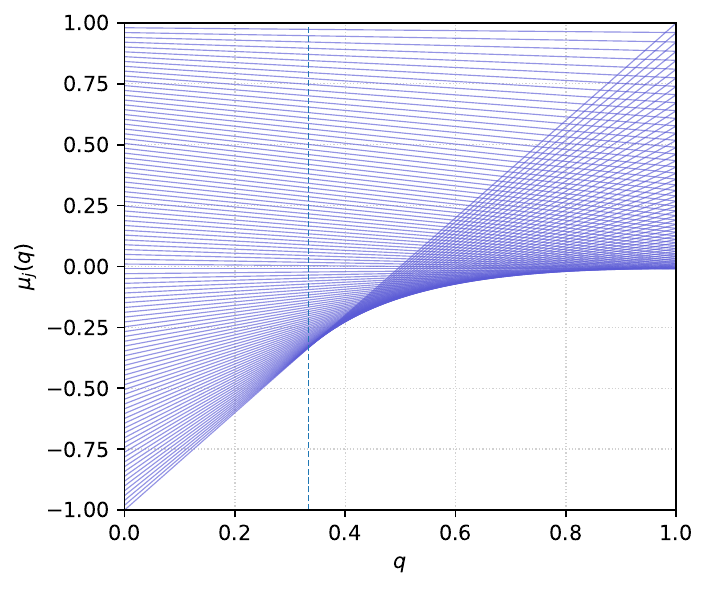}
  \end{minipage}\hfill
  \begin{minipage}[t]{0.49\textwidth}
    \centering
    \includegraphics[width=\textwidth]{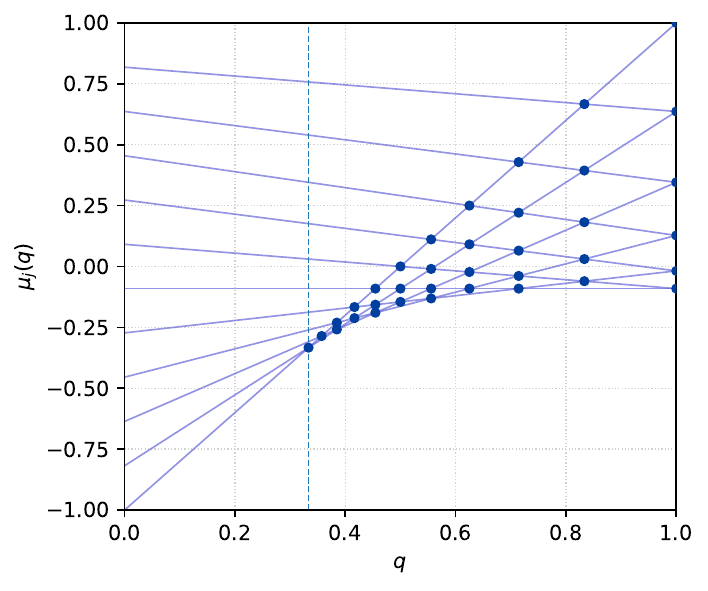}
  \end{minipage}
  \caption{Eigenvalue curves $\mu_j(q)=\Theta(\lambda_j)$ for $N=101$
  (left) and $N=11$ (right). The vertical dotted line marks $q=1/3$, the
  first parameter at which a collision occurs. The blue points in the right
  panel mark double eigenvalues.}
  \label{fig:eigenvalues}
\end{figure}

\subsection{The associated matrix valued orthogonal polynomials}
\label{sec:MVOP-Ehrenfestq}
Assume that $N$ is odd and $q>0$, and set $r=(N+1)/2$. We now specialize
the finite construction of Proposition~\ref{prop:finite-scalar-to-matrix} to
$M_q=\Theta(M_0)$. In the normalization used in that construction, the
scalar family $(p_n)_{0\leq n\leq N}$ satisfies
\[
   tp_n(t)=\frac{N-n}{N}p_{n+1}(t)+\frac{n}{N}p_{n-1}(t),
   \qquad p_{-1}=0,\quad p_0=1.
\]
Comparison with \eqref{eq:recurrence_Krawtchouk} gives
\begin{equation}
\label{eq:p_nKrawtchouk}
   p_n(t)=K_n\left(\frac{N(1-t)}{2}\right),
   \qquad
   p_n(\lambda_x)=K_n(x),
   \qquad
   \lambda_x=1-\frac{2x}{N}.
\end{equation}
Thus \eqref{eq:Krawtchouk-orthogonality} becomes
\begin{equation}
\label{eq:ehrenfest-scalar-orthogonality}
   \sum_{x=0}^N p_n(\lambda_x)p_m(\lambda_x)w(x)
   =\frac{\delta_{n,m}}{\binom{N}{n}},
   \qquad
   w(x)=\binom{N}{x}2^{-N},
   \qquad 0\leq n,m\leq N.
\end{equation}

Since $p_0(t)=1$ and $p_1(t)=t$, the polynomial part of the matrix weight is
\begin{equation}
\label{eq:Krawtchouk-2x2-weight}
   W_p(t)
   =\begin{pmatrix}1&t\\ t&t^2\end{pmatrix}.
\end{equation}
At the scalar spectral node $\lambda_x$, define
\begin{equation}
\label{eq:ehrenfest-atomic-weight}
   W(x)=w(x)W_p(\lambda_x)
   =w(x)\begin{pmatrix}
      1&\dfrac{N-2x}{N}\\[1ex]
      \dfrac{N-2x}{N}&\dfrac{(N-2x)^2}{N^2}
   \end{pmatrix}.
\end{equation}
The discrete push-forward matrix measure is therefore
\begin{equation}
\label{eq:ehrenfest-matrix-measure}
   \Sigma_q=\sum_{x=0}^N W(x)\,\delta_{\Theta(\lambda_x)}.
\end{equation}

Let $(P_k)_{0\leq k\leq r-1}$ be the finite family determined by the block
recurrence of $M_q$ in \eqref{eq:Mq_Ehrenfestmatricial}. The finite
scalar-to-matrix identity is
\begin{equation}
\label{eq:ehrenfest-matrix-to-scalar}
   P_k(\Theta(\lambda_x))
   \begin{pmatrix}1\\ \lambda_x\end{pmatrix}
   =\begin{pmatrix}
      p_{2k}(\lambda_x)\\ p_{2k+1}(\lambda_x)
   \end{pmatrix},
   \qquad
   0\leq k\leq r-1,\quad 0\leq x\leq N.
\end{equation}
Therefore,
\begin{equation}
\label{eq:ehrenfest-matrix-orthogonality}
   \sum_{x=0}^N
   P_j(\Theta(\lambda_x))W(x)
   P_k(\Theta(\lambda_x))^\ast
   =\delta_{j,k}H_j,
\end{equation}
where
\begin{equation}
\label{eq:ehrenfest-norm-matrices}
   H_j=\operatorname{diag}\left(
      \binom{N}{2j}^{-1},\binom{N}{2j+1}^{-1}
   \right),
   \qquad 0\leq j\leq r-1.
\end{equation}
Equations \eqref{eq:ehrenfest-matrix-measure} and
\eqref{eq:ehrenfest-matrix-orthogonality} give the matrix measure and its
orthogonality relations explicitly.

\begin{remark}
The matrices $W(x)$ are independent of $q$; the parameter enters through the
locations $\Theta(\lambda_x)$ of the atoms. When a spectral collision occurs,
the weights of all scalar nodes with the same image are added at the
corresponding atom of $\Sigma_q$. For $q=1/2$, this gives an explicit formula
for the orthogonality relations established abstractly in
\cite{grunbaum2009}*{p.~274}.
\end{remark}

The binomial distribution
\[
   \pi_i=\binom Ni2^{-N},\qquad 0\leq i\leq N,
\]
is stationary for every $q$. Indeed, choosing a ball or a pair of balls and
moving the selection to the opposite urn defines a symmetric walk on the
$2^N$ ball configurations. The uniform distribution on configurations is
therefore stationary, and its projection onto the number of balls in the
first urn is the binomial distribution $\pi$. For $0\leq q<1$, the one-ball
moves make the projected chain irreducible, so $\pi$ is its unique stationary
distribution. At $q=1$, parity is preserved: the even and odd states form two
closed classes, and $\pi$ is stationary but is no longer unique.

\subsection{Transition probabilities and the matrix Karlin--McGregor formula}
\label{subsec:ehrenfest-Karlin-McGregor}

The orthogonality measures constructed above also give spectral
representations of the transition probabilities. The scalar measure resolves
individual entries of the transition matrix, while its matrix push-forward
resolves the $m\times m$ blocks obtained by grouping consecutive states.

\begin{prop}
\label{prop:matrix-Karlin-McGregor}
Let $J$ be either the infinite Jacobi matrix of the general construction or
the complete Jacobi matrix of the finite construction, and suppose that both
$J$ and $\Theta(J)$ are stochastic. For $\ell\in\mathbb N_0$ and all
admissible scalar indices $r,s$,
\begin{equation}
\label{eq:KarlinMcgregor-banded}
   (\Theta(J)^\ell)_{r,s}
   =\frac{1}{h_s}\int_{\mathbb R}
   \Theta(x)^\ell p_r(x)p_s(x)\,d\mu(x).
\end{equation}
If $[\Theta(J)^\ell]_{i,j}$ denotes the $(i,j)$ block of size $m\times m$,
then, for all admissible block indices $i,j$,
\begin{equation}
\label{eq:KarlinMcGregor-matrix}
   [\Theta(J)^\ell]_{i,j}
   =\left(\int_{\mathbb R}\Theta(x)^\ell
   P_i(\Theta(x))W_p(x)P_j(\Theta(x))^\ast\,d\mu(x)\right)H_j^{-1}.
\end{equation}
\end{prop}

\begin{proof}
From $J\mathbf p(x)=x\mathbf p(x)$ we obtain
\[
   \Theta(J)^\ell\mathbf p(x)=\Theta(x)^\ell\mathbf p(x).
\]
The $r$-th component of this identity is
\[
   \sum_t(\Theta(J)^\ell)_{r,t}p_t(x)
   =\Theta(x)^\ell p_r(x).
\]
Here the sum extends over the scalar indices of $J$ and is finite when $J$ is
finite. Multiplying by $p_s(x)$ and integrating with respect to $\mu$,
scalar orthogonality leaves only the term $t=s$ and gives
\eqref{eq:KarlinMcgregor-banded} after division by $h_s$.

For the block formula, Propositions~\ref{prop:matrix-to-scalar} and
\ref{prop:finite-scalar-to-matrix}, in the infinite and finite settings,
respectively, give
\[
   P_i(\Theta(x))W_p(x)P_j(\Theta(x))^\ast
   =\mathbf p_i(x)\mathbf p_j(x)^\ast.
\]
Applying \eqref{eq:KarlinMcgregor-banded} entry by entry therefore yields
\[
   \int_{\mathbb R}\Theta(x)^\ell
   P_i(\Theta(x))W_p(x)P_j(\Theta(x))^\ast\,d\mu(x)
   =[\Theta(J)^\ell]_{i,j}H_j.
\]
Multiplying from the right by $H_j^{-1}$ proves
\eqref{eq:KarlinMcGregor-matrix}. This completes the proof of the
proposition.
\end{proof}

For the mixed Ehrenfest chain, $M_q=\Theta(M_0)$, the scalar nodes are
$\lambda_x=1-2x/N$, and $h_s=\binom{N}{s}^{-1}$. Thus, for
$\ell\in\mathbb N_0$ and $0\leq r,s\leq N$,
\begin{equation}
\label{eq:ehrenfest-KarlinMcGregor-scalar}
   (M_q^\ell)_{r,s}
   =\binom{N}{s}\sum_{x=0}^{N}
   \Theta(\lambda_x)^\ell
   p_r(\lambda_x)p_s(\lambda_x)w(x).
\end{equation}
The left-hand side is the $\ell$-step transition probability from state $r$
to state $s$. Formula \eqref{eq:ehrenfest-KarlinMcGregor-scalar} therefore
expresses this probability as a finite Krawtchouk spectral sum.

When $N$ is odd and $q>0$, group the states into the pairs
$\{2i,2i+1\}$. With
\[
   H_j=\operatorname{diag}
   \left(\binom{N}{2j}^{-1},\binom{N}{2j+1}^{-1}\right),
\]
for $0\leq i,j\leq(N-1)/2$, the corresponding $2\times2$ block formula is
\begin{equation}
\label{eq:ehrenfest-KarlinMcGregor-matrix}
   [M_q^\ell]_{i,j}
   =\left(
   \sum_{x=0}^{N}\Theta(\lambda_x)^\ell
   P_i(\Theta(\lambda_x))W(x)
   P_j(\Theta(\lambda_x))^\ast
   \right)H_j^{-1}.
\end{equation}
More explicitly, if $a,b\in\{0,1\}$, then
\[
   \bigl([M_q^\ell]_{i,j}\bigr)_{a,b}
   =\mathbb P\{X_\ell=2j+b\mid X_0=2i+a\}.
\]
Thus the matrix formula simultaneously gives the four transition
probabilities between the two states in $\{2i,2i+1\}$ and the two states in
$\{2j,2j+1\}$.

Formulas \eqref{eq:ehrenfest-KarlinMcGregor-scalar} and
\eqref{eq:ehrenfest-KarlinMcGregor-matrix} also make the role of $q$
transparent. Although the polynomials $P_i$ depend on $q$, the identity
\eqref{eq:ehrenfest-matrix-to-scalar} shows that
\[
   P_i(\Theta(\lambda_x))W(x)
   P_j(\Theta(\lambda_x))^\ast,
\]
is independent of $q$. Hence the parameter dependence in the spectral sum
enters through $\Theta(\lambda_x)^\ell$. In the push-forward measure, it also
determines the locations of the atoms and which weights are combined:
\[
   \Sigma_q(\{y\})
   =\sum_{\{x:\,\Theta(\lambda_x)=y\}}W(x).
\]
Thus spectral collisions are automatically incorporated. The same
representation applies to the $k$-ball and multi-ball chains below after
replacing $\Theta$ by the corresponding Krawtchouk polynomial or linear
combination of Krawtchouk polynomials.

\subsection{Relation with the  beefed-up Ehrenfest model}
\label{subsec:grunbaum-answers}

For the equal mixture $q=1/2$, Grünbaum groups the five-diagonal transition
matrix into $2\times2$ blocks, constructs the finite matrix recurrence, and
determines the eigenvalues and their multiplicities
\cite{grunbaum2009}*{pp.~272--273}. After symmetrizing the transition matrix,
he establishes nonnegative matrix weights and transition formulas and
expresses the weights through the resolvent and the terminal matrix
polynomials \cite{grunbaum2009}*{pp.~274--275}.

The identity $M_q=\Theta(M_0)$ supplies a scalar origin for this construction:
$M_0$ and $M_q$ share the Krawtchouk eigenvectors, and $\Theta$ maps the
classical spectrum to that of the mixed chain. The resulting multiplicities
are classified for every $q\in[0,1]$ in
Proposition~\ref{prop:ehrenfest-spectral-collisions}. For odd $N$ and $q>0$,
\eqref{eq:ehrenfest-matrix-to-scalar} gives a scalar realization of the matrix
polynomials, and \eqref{eq:ehrenfest-matrix-measure} identifies their measure
as the push-forward of the Krawtchouk measure. At $q=1/2$, this recovers
Grünbaum's measure in the normalization used here; at a collision, the
corresponding Krawtchouk weights are added. Finally, the formulas
\eqref{eq:ehrenfest-KarlinMcGregor-scalar} and
\eqref{eq:ehrenfest-KarlinMcGregor-matrix} provide the corresponding
transition probabilities throughout the family $q\in[0,1]$.

\section{Ehrenfest extensions: the $k$-ball and multi-ball models}
\label{sec:ehrenfest-extensions}

The polynomial description of the mixed one-/two-ball chain extends to moves
involving an arbitrary number of balls. We first treat a fixed number $k$ of
balls and then combine several such dynamics in a multi-ball mixture.

\subsection{The $k$-ball model}
\label{subsec:kball-ehrenfest}

Fix $k\in\{1,\ldots,N\}$. At each step, choose $k$ of the $N$ balls
uniformly at random and move every chosen ball to the opposite urn. As before,
let $X_n$ be the number of balls in urn $A$ after $n$ moves. At the
configuration level, this dynamics is the non-local walk on the hypercube
that flips $k$ uniformly chosen coordinates, studied from the viewpoint of
mixing times and cutoff in \cite{nestoridi2017nonlocal}.
If the current state is $i$ and the
selection contains $k-\ell$ balls from urn $A$ and $\ell$ balls from urn $B$,
then the next state is $i-k+2\ell$. Hence
\[
   \mathbb P\{X_{n+1}=j\mid X_n=i\}
   =\begin{cases}
      \displaystyle
      \frac{\binom{i}{k-\ell}\binom{N-i}{\ell}}{\binom{N}{k}},
      &j=i-k+2\ell,\quad 0\leq\ell\leq k,\\[2ex]
      0,&\text{otherwise},
   \end{cases}
\]
where $\binom ab=0$ when $b\notin\{0,\ldots,a\}$. The possible increments
are $-k,-k+2,\ldots,k$. In particular, the next state has the same parity as
$i+k$, and the transition matrix has $k+1$ potentially nonzero diagonals:
\begin{equation}
\label{eq:transitionM-kball}
   J_k=\sum_{i=0}^N
   \sum_{\substack{0\leq\ell\leq k\\0\leq i-k+2\ell\leq N}}
   \frac{\binom{i}{k-\ell}\binom{N-i}{\ell}}{\binom{N}{k}}
   E_{i,i-k+2\ell}.
\end{equation}
\begin{remark}
The case $k=1$ is the classical Ehrenfest chain, so $J_1=M_0$. The mixed
model of Subsection~\ref{sec:Ehrefest-with-q} chooses between the one- and
two-ball dynamics and therefore has transition matrix
\[
   M_q=(1-q)J_1+qJ_2.
\]
\end{remark}
\begin{prop}
\label{prop:Jk-recurrence}
The matrices $J_k$ satisfy the recurrence relation
\begin{equation}
\label{eq:recurrence_Jk}
   (k-N)J_{k+1}=kJ_{k-1}-NJ_1J_k,
   \qquad 1\leq k\leq N-1,
\end{equation}
with the initial condition $J_0=I$ and $J_1$ given by \eqref{eq:M_0}.
\end{prop}
\begin{proof}
Fix $0\leq i\leq N$. Both sides vanish in row $i$ outside columns of the
form $j=i-(k+1)+2\ell$, where $0\leq\ell\leq k+1$. For such a column, the
tridiagonal form of $J_1$ gives
\begin{equation}
\label{eq:recurrence_right}
\begin{aligned}
&(kJ_{k-1}-NJ_1J_k)_{i,i-(k+1)+2\ell}\\
&\quad=k(J_{k-1})_{i,i-(k-1)+2(\ell-1)}
-i(J_k)_{i-1,i-1-k+2\ell}\\
&\qquad\quad
-(N-i)(J_k)_{i+1,i+1-k+2(\ell-1)}.
\end{aligned}
\end{equation}
The convention for binomial coefficients makes this formula valid at the
boundary values of $i$ and $\ell$. Substitution from
\eqref{eq:transitionM-kball} gives
\begin{multline*}
\frac{k\binom{i}{k-\ell}\binom{N-i}{\ell-1}}
     {\binom{N}{k-1}}
-\frac{i\binom{i-1}{k-\ell}\binom{N-i+1}{\ell}}
      {\binom{N}{k}}\\
-\frac{(N-i)\binom{i+1}{k-\ell+1}\binom{N-i-1}{\ell-1}}
      {\binom{N}{k}}
=(k-N)\frac{\binom{i}{k-\ell+1}\binom{N-i}{\ell}}
              {\binom{N}{k+1}}.
\end{multline*}
The last expression is
$(k-N)(J_{k+1})_{i,i-(k+1)+2\ell}$, proving
\eqref{eq:recurrence_Jk} entry by entry.
\end{proof}

Recurrence \eqref{eq:recurrence_Jk} generates every $J_k$ from $J_1$ and
therefore identifies all the fixed-size dynamics as polynomial transforms of
the classical chain.
\begin{cor}
\label{cor:k_ball_matrix}
For $0\leq k\leq N$,
\begin{equation}
\label{eq:J_k-as-Krawtchouk}
   J_k=K_k\left(-\frac{N(J_1-I)}{2}\right).
\end{equation}
Hence, the Krawtchouk eigenvectors $v_j$ in
\eqref{eq:lambda_j_Krawtchouk} satisfy
\begin{equation}
\label{eq:Jk-Krawtchouk-eigenvalues}
   J_kv_j=K_k(j)v_j,
   \qquad 0\leq j\leq N.
\end{equation}
\end{cor}
\begin{proof}
Define scalar polynomials $(R_k)_{0\leq k\leq N}$ by $R_0(t)=1$,
$R_1(t)=t$, and the recurrence obtained from
\eqref{eq:recurrence_Jk}. Equivalently,
\[
   tR_k(t)=\frac{N-k}{N}R_{k+1}(t)
           +\frac{k}{N}R_{k-1}(t),
   \qquad 1\leq k\leq N-1.
\]
Equation \eqref{eq:p_nKrawtchouk} shows that
\[
   R_k(t)=K_k\left(\frac{N(1-t)}{2}\right).
\]
Proposition~\ref{prop:Jk-recurrence} and the initial values now give
$J_k=R_k(J_1)$, which is \eqref{eq:J_k-as-Krawtchouk}. Finally,
$J_1v_j=\lambda_jv_j$ and $N(1-\lambda_j)/2=j$, so functional calculus gives
\eqref{eq:Jk-Krawtchouk-eigenvalues}.
\end{proof}

\subsection{Multi-ball mixtures}
\label{subsec:multi-ball-ehrenfest}

We now allow the number of balls moved at a given step to be random. Fix
$s\in\mathbb N$ and let
\begin{equation}
\label{eq:vectors-q-k}
   \mathbf q=(q_1,\ldots,q_s),
   \qquad
   \mathbf k=(k_1,\ldots,k_s),
\end{equation}
where $q_a\geq0$, $\sum_{a=1}^s q_a=1$, and
$k_a\in\{0,\ldots,N\}$. At each step, choose an index $a$ with probability
$q_a$ and perform a $k_a$-ball move. Here $k_a=0$ represents a holding move,
with $J_0=I$. The transition matrix is the convex combination
\begin{equation}
\label{eq:multi-ball-transition-matrix}
\begin{aligned}
   J_{\mathbf q,\mathbf k}
   &=\sum_{a=1}^s q_aJ_{k_a}\\
   &=\sum_{a=1}^s q_a
      K_{k_a}\left(-\frac{N(J_1-I)}{2}\right)
    =\Theta_{\mathbf q,\mathbf k}(J_1),
\end{aligned}
\end{equation}
where
\begin{equation}
\label{eq:multi-ball-polynomial}
   \Theta_{\mathbf q,\mathbf k}(t)
   =\sum_{a=1}^s q_a
      K_{k_a}\left(\frac{N(1-t)}{2}\right).
\end{equation}
Since $J_1v_j=\lambda_jv_j$, the eigenvectors remain the Krawtchouk vectors
$v_j$, and the corresponding eigenvalues are
\begin{equation}
\label{eq:multi-ball-eigenvalues}
   \Theta_{\mathbf q,\mathbf k}(\lambda_j)
   =\sum_{a=1}^s q_aK_{k_a}(j),
   \qquad 0\leq j\leq N.
\end{equation}

The binomial distribution $\pi=(\pi_0,\ldots,\pi_N)$ is stationary for every
such mixture because it is stationary for each $J_{k_a}$.
\begin{remark}
The classical chain corresponds to $(\mathbf q,\mathbf k)=((1),(1))$, the
fixed $k$-ball chain to $((1),(k))$, and the mixed one- and two-ball model to
$((1-q,q),(1,2))$.
\end{remark}
\begin{example}
\label{ex:multiball-fig}
Take $\mathbf q=(1-q,q)$ and $\mathbf k=(1,k)$, where $0\leq q\leq1$ and
$0\leq k\leq N$. At each step, the chain makes a classical one-ball move
with probability $1-q$ and a $k$-ball move with probability $q$. Its
transition matrix is
\[
   J_{q,k}=(1-q)J_1+qJ_k=\Theta_{q,k}(J_1),
\]
where
\[
   \Theta_{q,k}(t)
   =(1-q)K_1\left(\frac{N(1-t)}{2}\right)
     +qK_k\left(\frac{N(1-t)}{2}\right).
\]
The possible nonzero diagonal offsets of $J_{q,k}$ belong to
$\{-1,1\}\cup\{-k,-k+2,\ldots,k\}$. Its eigenvalues are
\begin{equation}
\label{eq:lambda_multi_ball_example}
   \mu_j^{(k)}(q)=(1-q)K_1(j)+qK_k(j),
   \qquad 0\leq j\leq N.
\end{equation}
\end{example}
Figure~\ref{fig:eigenvalues_in_k} fixes $N=101$ and displays the eigenvalues
\eqref{eq:lambda_multi_ball_example} as functions of the integer parameter
$k$. It also shows the largest eigenvalue modulus strictly below $1$,
\[
   \widehat\rho_{q,k}
   =\max_{\substack{0\leq j\leq N\\|\mu_j^{(k)}(q)|<1}}
      |\mu_j^{(k)}(q)|.
\]
For odd $k$, this excludes both the stationary eigenvalue $1$ and the parity
eigenvalue $-1$.

\begin{figure}[ht]
  \centering
  \begin{minipage}[t]{0.49\textwidth}
    \centering
    \includegraphics[width=\textwidth]{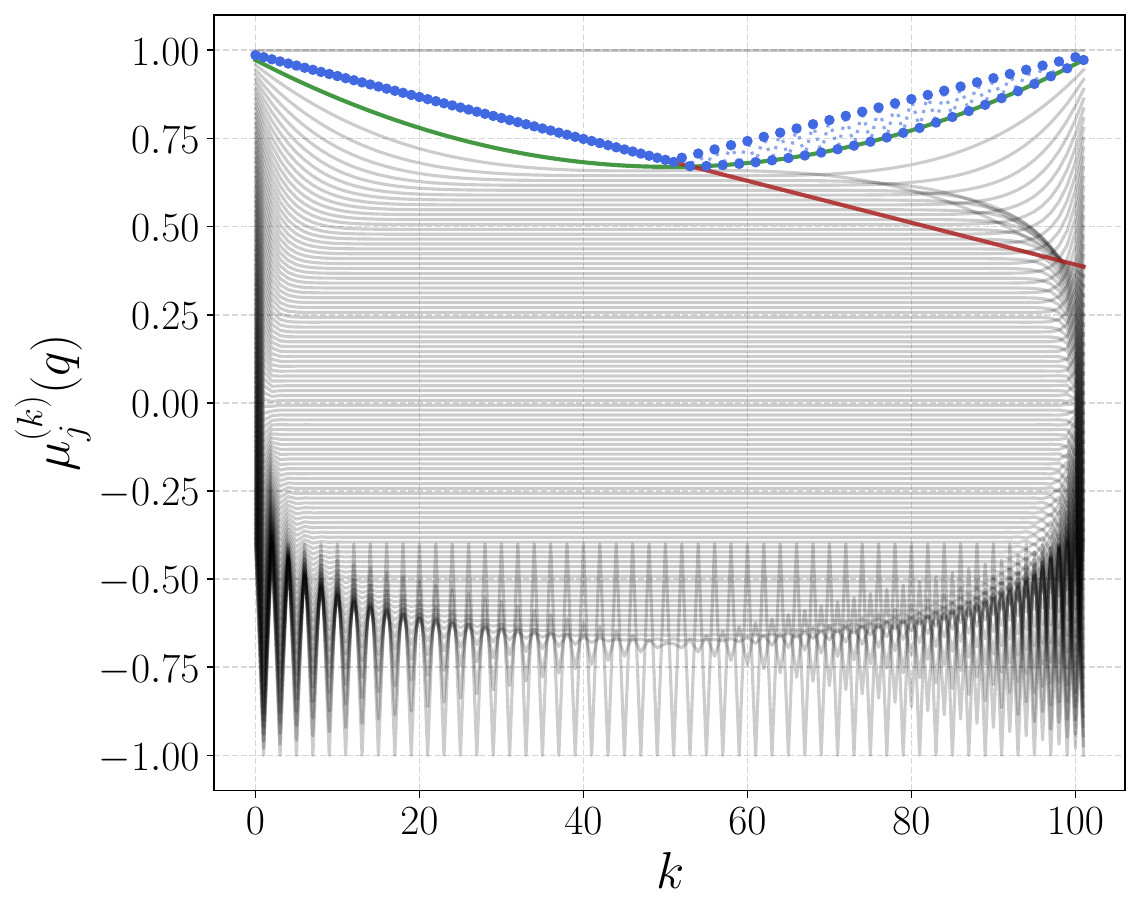}
  \end{minipage}\hfill
  \begin{minipage}[t]{0.49\textwidth}
    \centering
    \includegraphics[width=\textwidth]{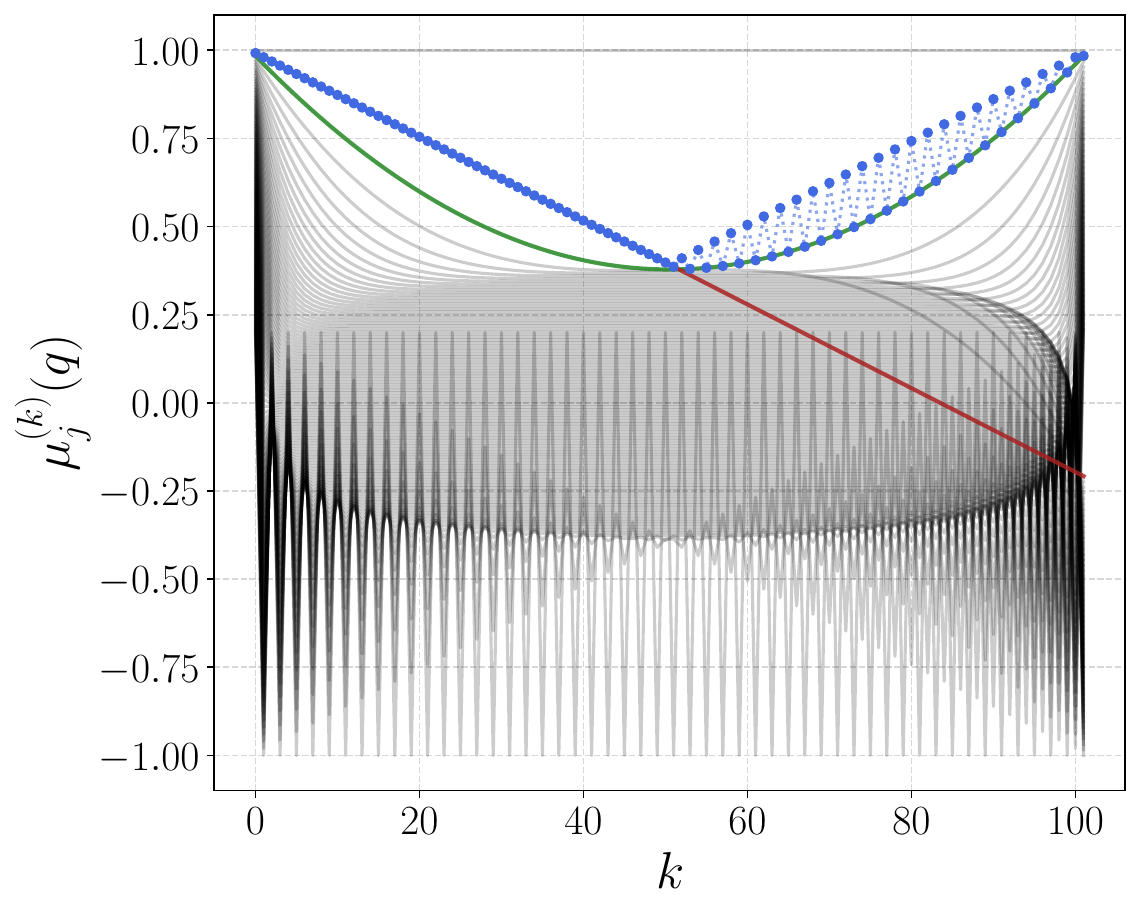}
  \end{minipage}
  \caption{Eigenvalues $\mu_j^{(k)}(q)$ as functions of $k$ for $N=101$,
  with $q=0.3$ (left) and $q=0.6$ (right). The red and green curves
  correspond to $j=1$ and $j=2$, respectively; the remaining eigenvalue
  curves are gray. The blue dots show the largest eigenvalue modulus strictly
  below $1$, denoted by $\widehat\rho_{q,k}$.}
  \label{fig:eigenvalues_in_k}
\end{figure}

\section{Conclusions}
\label{sec:conclusions}

Starting from a scalar orthogonal polynomial family with Jacobi matrix $J$,
we have studied the matrix valued orthogonal polynomials associated with the
block Jacobi matrix $\Theta(J)$. The scalar-to-matrix identity proved in this
paper gives a direct realization of the matrix polynomials in terms of their
scalar counterparts. We identified their measure of orthogonality as a
push-forward matrix measure, proved that the corresponding matrix inner
product is non-degenerate, and derived an inverse-branch formula for its
density. The finite version of the construction includes the terminal
recurrence and the precise truncated module on which non-degeneracy holds.

Reducibility must be analyzed intrinsically in the push-forward variable. The
cubic transported-Legendre example shows that changes in the pointwise rank
of the density do not imply reducibility. The quadratic Hermite family
exhibits a different phenomenon: it is irreducible for
$\alpha_1\neq0$, whereas for $\alpha_1=0$ it decomposes into two scalar
Laguerre families.

For stochastic polynomial transforms, the construction provides both scalar
and block Karlin--McGregor representations of the transition probabilities.
In the mixed one- and two-ball Ehrenfest model, the identity
$M_q=\Theta(M_0)$ yields the Krawtchouk eigenbasis, the explicit eigenvalues,
the complete classification of spectral collisions for $q\in[0,1]$, and the
matrix measure as a finite push-forward of the Krawtchouk measure. For the
fixed $k$-ball model, we proved that the transition matrix is
\[
   J_k=K_k\left(-\frac{N(J_1-I)}{2}\right),
\]
and multi-ball mixtures are then described by linear combinations of these
polynomials. Their eigenvectors remain the Krawtchouk vectors, while their
eigenvalues are obtained by evaluating the corresponding polynomial
transform on the classical Ehrenfest spectrum.

These results suggest two directions for further study. On the probabilistic
side, the explicit eigenvalues may be used to investigate convergence to
stationarity, mixing times, and possible cutoff phenomena when
$N\to\infty$ and the number of balls moved depends on $N$. On the structural
side, it is natural to ask which other stochastic models admit polynomial
representations of this kind and whether their scalar spectral data similarly
produce explicit matrix valued orthogonal polynomials.

\end{document}